\documentclass[12pt]{amsart}

\usepackage{amsmath,amsthm,amsfonts,amssymb,amscd,euscript}
\usepackage{enumerate,enumitem}
\usepackage{color}
\usepackage{graphics}
\usepackage{tikz}
\usepackage{tkz-graph}
\usepackage{tkz-euclide}
\usetkzobj{polygons}
\usepackage[all]{xy}
\usepackage{mathtools}
\usepackage{tkz-graph}
\usetikzlibrary{automata,arrows,patterns,shapes.geometric,automata,arrows}
\usepackage{graphicx}
\usepackage{todonotes}
\usepackage{hyperref}
\usepackage{marginnote}
\usepackage{bookmark}
\newlength{\defbaselineskip}
\advance \topmargin by -\headheight
\advance \topmargin by -\headsep
\evensidemargin \oddsidemargin
\theoremstyle{plain}
\newtheorem{theorem}{Theorem}[section]
\newtheorem{lemma}[theorem]{Lemma}

\newtheorem{corollary}[theorem]{Corollary}
\newtheorem{conj}{Conjecture}

\theoremstyle{definition}
\newtheorem{definition}[theorem]{Definition}

\DeclareMathOperator{\mut}{Mut}

\DeclareMathOperator{\spec}{Spec}

\renewcommand{\emptyset}{\varnothing}
\newcommand{\mgsi}{\ensuremath{\fontseries\bfdefault\textit{\underline{i}}}}
\newcommand{\hdef}[1]{\textbf{#1}}
\newcommand{\A}{\mathcal{A}}
\newcommand{\U}{\mathcal{U}}
\newcommand{\ZZ}{\mathcal{Z}}

\newcommand{\Z}{\mathbb{Z}}
\newcommand{\T}{\mathbb{T}}
\renewcommand{\O}{\mathrm{O}}
\newcommand{\x}{\normalfont\textbf{x}}

\newcommand{\F}{\mathcal{F}}

\renewcommand{\tilde}{\widetilde}

\newcommand{\ou}{%
  \mathrel{%
    \vcenter{\offinterlineskip
      \ialign{##\cr$\cup$\cr\noalign{\kern+1.5pt}$\cap$\cr}%
    }%
  }%
}

\definecolor{darksagegreen}{RGB}{0, 105, 60} 

\tikzset{inner sep = 1pt}
\tikzset{outer sep = 1pt}
\tikzset{vertex/.style = {shape=circle,draw,minimum size=1.2em, ultra thick}}
\tikzset{fvertex/.style = {shape=rectangle,draw,minimum size=2em, ultra thick}}
\tikzset{qarrow/.style = {->, shorten <=5, shorten >=5, >=stealth', thick}}

\begin{document}

\title[green-to-red sequences, local-acyclicity, and upper cluster algebras]{On the relationship between green-to-red sequences, local-acyclicity, and upper cluster algebras}
\author{Matthew R. Mills}
\thanks{The author was supported by the University of Nebraska--Lincoln.}
\keywords{upper cluster algebra, maximal green sequences, locally-acyclic}
\address{Department of Mathematics, 
University of Nebraska-Lincoln,
Lincoln NE 68588-0130,
USA}
\email{matthew.mills@huskers.unl.edu}

\subjclass[2010]{Primary 13F60 
} 

\begin{abstract}
A cluster is a finite set of generators of a cluster algebra. The Laurent Phenomenon of Fomin and Zelevinsky says that any element of a cluster algebra can be written as a Laurent polynomial in terms of any cluster. The upper cluster algebra of a cluster algebra is the ring of rational functions that can be written as a Laurent polynomial in every cluster of the cluster algebra. By the Laurent phenomenon a cluster algebra is always contained in its upper cluster algebra, but they are not always equal.

In 2014 it was conjectured that the equality of the cluster algebra and upper cluster algebra is equivalent to a combinatorial property regarding the existence of a maximal green sequence. In this work we prove a stronger result for cluster algebras from mutation-finite quivers, and provide a counterexample to show that the conjecture does not hold in general. Finally, we propose a new conjecture about the upper cluster algebra on the relationship which replaces maximal green sequences with more general green-to-red sequences and incorporates Mueller's local-acyclicity. 
\end{abstract}
\maketitle
\section{Introduction}

Cluster algebras were introduced by Fomin and Zelevinsky in the early 2000s \cite{cluster1,cluster2,cluster3,cluster4} as a way to study total positivity and (dual) canonical bases in Lie theory. Cluster algebras have since become fundamental objects of study in many areas of mathematics such as commutative algebra, algebraic geometry, algebraic combinatorics, representation theory, and mathematical physics. They play a role in the study of Teichm\"{u}ller theory, canonical bases, total positivity, Poisson-Lie groups, Calabi-Yau algebras, noncommutative Donaldson-Thomas invariants, scattering amplitudes, and representations of finite dimensional algebras.

A cluster algebra is a subalgebra of a rational function field with a distinguished set of generators, called \emph{cluster variables}, that are obtained by an iterative procedure, called \emph{mutation}, and grouped into overlapping finite subsets of a fixed cardinality, called \emph{clusters}. The dynamics of the mutation process are determined by a directed graph with no loops or two-cycles which is called a \emph{(cluster) quiver}. The number of vertices in the quiver is the same as the size of the cluster and the cluster variables are associated to the vertices of the quiver. The procedure of mutation requires a designated vertex $v$ of the quiver and then acts on both the cluster and the quiver. In the quiver, mutation changes the edges of the digraph to produce a new one. In the cluster, the cluster variable associated with $v$ is replaced by a new one which is a rational function in the cluster variables associated to vertices adjacent to $v$. 

The Laurent Phenomenon (Theorem~\ref{thm:Laurent}) says that all cluster variables can be written as Laurent polynomials in terms of any cluster. However, there could still exist Laurent polynomials with this property that are not elements of the cluster algebra. The \textit{upper cluster algebra} $\mathcal{U}$ associated to a cluster algebra $\mathcal{A}$ is the intersection of the Laurent rings of every cluster of $\mathcal{A}.$ This means that $\mathcal{U}$ is the subalgebra of all rational functions that can be written as Laurent polynomials in terms of any cluster of $\mathcal{A}$. In general $\mathcal{U}$ is a much nicer ring then $\mathcal{A}$ \cite{Gross2015,gross}. For example, $\U$ is often finitely generated where $\A$ is not, and $\U$ has better behaved singularities than $\A$ \cite{Benito2015}. By Theorem~\ref{thm:Laurent} we have a natural inclusion $\mathcal{A} \subset \mathcal{U},$ but in general it remains an open question as to when this containment is strict.

When a quiver has no directed cycles it is called \emph{acyclic}. A cluster algebra is said to be acyclic if it is generated from an acyclic quiver and it is known that an acyclic cluster algebra always coincides with its upper cluster algebra \cite{cluster3,MullerAU}. In \cite{muller} Muller generalizes the notion of an acyclic cluster algebra and defines \emph{locally-acyclic} cluster algebras.

If we consider a cluster algebra $\A$ with cluster variables $x_1$ and $x_2$ which have no common zeros in $\spec \A$, then $\spec \A[x_1^{-1}]$ and $\spec\A[x_2^{-1}]$ form an open cover of $\spec \A.$ A cluster algebra then said to be locally-acyclic if there exist collections $Z_i$ of cluster variables such that when we consider the localizations $\A_i = \A[Z_i^{-1}]$ we have that each $\A_i$ is an acyclic cluster algebra and the collection $\{\spec \A_i\}$ form an open cover of $\spec \A.$
The class of locally-acyclic cluster algebras naturally contain the class of  acyclic cluster algebras and they share many of the same nice properties; including that the cluster algebra and upper cluster algebra coincide. A precise definition of a locally-acyclic cluster algebra can be found in Section~\ref{sec:laca}.

A \emph{maximal green sequence} is a particular finite sequence of mutations such that each mutation occurs at a vertex satisfying a certain combinatorial condition that was introduced by Keller \cite{keller2} to produce quantum dilogarithm identities and explicitly compute refined Donaldson-Thomas invariants defined by Kontsevich and Soibelman \cite{Kontsevich2008}. They are also used by physicists to count BPS states \cite{Alim2013a}. For a precise definition we refer the reader to  \cite{brustle}.

During a discussion led by Arkady Berenstein and Christof Geiss at the \textit{Hall and cluster algebras} conference at Centre de Recherches Math\'ematiques, Montr\'eal in 2014 it was observed that in every known example a cluster algebra is equal to its upper cluster algebra if and only if the cluster algebra contains a quiver that admits a maximal green sequence. It was also then conjectured that this phenomenon was true in general. However, as Muller pointed out in \cite{muller2} there is some subtlety in making this conjecture precise. Many different quivers can produce the same cluster algebra, but the question about the existence of a maximal green sequence is specific to a single quiver. Therefore we state the conjecture as follows. 

\begin{conj}\label{conj:false}
Let $Q$ be a quiver. The following are equivalent:
\begin{enumerate}
\item The quiver $Q$ can be mutated to a quiver which admits a maximal green sequence.
\item The cluster algebra $\A(Q)$ is equal to its upper cluster algebra.
\end{enumerate}
\end{conj}

The purpose of this note is twofold. First, we will show that a stronger result holds for the class of mutation-finite quivers. Secondly, we will show that Conjecture~\ref{conj:false} does not hold in general by providing a counterexample. We will also give evidence for a new conjecture about the relationship between maximal green sequences, local-acyclicity, and upper cluster algebras. 

The \emph{mutation class} of a quiver $Q$, denoted $\mut(Q)$, is the collection of all quivers produced by mutating $Q$ in every possible way. If $\mut(Q)$ is finite then $Q$ is said to be \emph{mutation-finite}. A classification of mutation-finite quivers was given by Felikson, Shapiro, and Tumarkin \cite{FST}. Although these mutation classes contain finitely many quivers the mutation of seeds usually produce infinitely many clusters.

It has been an important open problem in string theory to determine which finite mutation classes have quivers that admit maximal green sequences.
Alim \textit{et al}. showed for any quiver from a marked surface with boundary had a maximal green sequence \cite{Alim2013a}. Bucher showed that a specific quiver from a closed surface with 2 punctures and genus at least 2 has a maximal green sequence \cite{bucher}. Ladkani showed that any quiver from a closed marked surface with exactly one marked point does not have a maximal green sequence \cite{ladkani}. Seven proved that the quivers in the mutation class of $\mathbb{X}_7$ does not have a maximal green sequence \cite{seven}.
Together with Eric Bucher, the author answered the question for the only infinite family of mutation classes that were still outstanding \cite{Bucher2017} by providing a single quiver in the mutation class with a maximal green sequence.

However, shortly after this result Muller showed in \cite{muller2} that the property of having a maximal green sequence is not an invariant of the mutation class. That is, it is possible for there to exist two quivers in the same mutation class where one has a maximal green sequence and the other does not. In \cite{mills} the author completely answers the question of existence of a maximal green sequence for every individual mutation-finite quiver (Theorem~\ref{thm:1}).

We cite several results where facts about the upper cluster algebra and the local-acyclicity are already known. Berenstein, Fomin and Zelevinsky show that the cluster algebra and upper cluster algebra coincide for acyclic cluster algebras when they contain a coprime seed \cite{cluster3}. Muller removed the assumption of containing a coprime seed when he showed the same result for locally-acyclic cluster algebras. It is known but, to the authors knowledge, never written that the cluster algebra from $\mathbb X_7$ is not equal to its upper cluster algebra. We give a proof of this result in Theorem~\ref{thm:LA_exceptional}. 
 Muller showed that all of the non-acyclic mutation-finite quivers of exceptional type except for $\mathbb X_7$ are locally-acyclic \cite{muller}. 
 Musiker, Schiffler, and Williams show that for any unpunctured marked surface with boundary and at least two marked points is equal to its upper cluster algebra \cite{MSW2}. 
 Canackci, Lee, and Schiffler reduced the number of necessary marked points to one \cite{cls}. 
Ladkani showed that the cluster algebra for a once-punctured closed surface is not equal to its upper cluster algebra \cite{ladkani}. We prove the outstanding cases to obtain the following results. 
{%
\begin{theorem}\label{thm:big1}For every mutation-finite quiver $Q$ there exists an open set $\O$ such that the cluster algebra with principal coefficients $\A_\O^\bullet(Q)$ is a locally-acyclic cluster algebra except for $\mathbb X_7$ and those from once-punctured closed surfaces. In the case of $\mathbb X_7$ and once-punctured closed surfaces the cluster algebras are not equal to the upper cluster algebra.
\end{theorem}
}

\begin{theorem}\label{thm:big2}
Let $Q$ be a mutation finite quiver with $n$ vertices. The following are equivalent:
\begin{enumerate}
\item The quiver $Q$ has a maximal green sequence.
\item There exists an open set $\O \subseteq \spec(\Z[y_1^{\pm 1},\dots,y_n^{\pm 1}])$ such that the cluster algebra generated over the the ring of regular functions on $\O$ with principal coefficients is equal to its upper cluster algebra.
\item There exists an open set $\O \subseteq \spec(\Z[y_1^{\pm 1},\dots,y_n^{\pm 1}])$ such that the cluster algebra generated over the the ring of regular functions on $\O$ with principal coefficients is locally-acyclic.
\end{enumerate}
Additionally, the stated properties are satisfied for all mutation-finite quivers except in the case of mutation classes associated to once-punctured surfaces and one other mutation class, called $\mathbb{X}_7$, which consists of two quivers.
\end{theorem}
\begin{proof}
This result follows from Theorem~\ref{thm:big1}, Theorem~\ref{thm:1}, and Theorem~\ref{thm:LA_AU}.
\end{proof}

Muller showed that in general the cluster algebra is equal to the upper cluster algebra when it is locally-acyclic \cite{muller}. That is, the third condition implies the second condition. This is the only known direct implication of the three conditions above.
To obtain our results we prove or disprove each condition for every quiver under consideration.

Amongst many other results in \cite{gross} Gross, Hacking, Keel, and Kontsevich show that the Fock-Goncharov conjecture about the existence of a canonical basis for cluster algebras holds when a cluster algebra has a quiver with a maximal green sequence and the cluster algebra is equal to its upper cluster algebra.

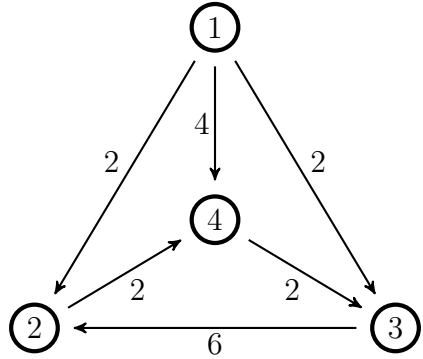
\begin{figure}
\tikzset{vertex/.style = {shape=circle,draw,minimum size=1.5em, ultra thick}}
\tikzset{fvertex/.style = {shape=rectangle,draw,minimum size=2em, ultra thick}}
\tikzset{qarrow/.style = {->, shorten <=5, shorten >=5, >=stealth', thick}}

\begin{tikzpicture}[scale = 0.8]
\node[vertex] (A) at (0,0){2};
\node[vertex] (B) at (6,0){3};
\node[vertex] (C) at (3,5){1};
\node[vertex] (D) at (3,1.8){4};

\draw[qarrow] (C) to node [above left]{ 2} (A); 
\draw[qarrow] (C) to node [above right]{ 2} (B); 
\draw[qarrow] (A) to node [below right]{ 2} (D); 
\draw[qarrow] (D) to node [below left]{ 2} (B); 
\draw[qarrow] (B) to node [below]{ 6} (A); 
\draw[qarrow] (C) to node [left]{ 4} (D);

\end{tikzpicture}

\caption{A quiver which provides a counterexample to Conjecture~\ref{conj:false}. We refer to this quiver as $Q_{ce}.$}\label{fig:counterexample}
\end{figure}
The quiver from Figure~\ref{fig:counterexample} which we will refer to as $Q_{ce}$ provides a counterexample to Conjecture~\ref{conj:false}. One interesting observation about $Q_{ce}$ is that it does admit a \emph{green-to-red sequence}. These mutation sequences are a generalization of maximal green sequences. Furthermore, it is shown in \cite{muller2} that the existence of a green-to-red sequence for a quiver $Q$ implies the existence of a green-to-red sequence for every quiver in $\mut(Q)$. 
\begin{theorem}\label{thm:big3}
 There is no quiver in $\mut(Q_{ce})$ which admits a maximal green sequence. The cluster algebra $\A(Q_{ce})$ is locally-acyclic, and hence equal to its upper cluster algebra. Every quiver in $\mut(Q_{ce})$ admits a green-to-red sequence. 
\end{theorem}
\begin{proof}
The result follows from Lemma~\ref{lem:ce_1}, Lemma~\ref{thm:LA_AU}, Corollary~\ref{cor:ce_1}, and Corollary~\ref{cor:ce_2}.
\end{proof}

We offer the following conjecture to replace Conjecture~\ref{conj:false} and subsequently provide further evidence for its validity. 

\begin{conj}\label{conj:big}
Let $Q$ be a quiver with $n$ vertices. The following are equivalent:
\begin{enumerate}
\item The quiver $Q$ has a green-to-red sequence.
\item There exists an open set $\O \subseteq \spec(\Z[y_1^{\pm 1},\dots,y_n^{\pm 1}])$ such that the cluster algebra generated over the the ring of regular functions on $\O$ with principal coefficients is equal to its upper cluster algebra.
\item There exists an open set $\O \subseteq \spec(\Z[y_1^{\pm 1},\dots,y_n^{\pm 1}])$ such that the cluster algebra generated over the the ring of regular functions on $\O$ with principal coefficients is locally-acyclic.
\end{enumerate}
Furthermore, when any of these properties hold, the same open set can be used for (2) and (3). 
\end{conj}

Although not explicitly stated in the work, John Lawson and the author prove Conjecture~\ref{conj:big} for a class of quivers called minimal mutation-infinite quivers \cite{lawson}. In an announced work Yakimov and Goodearl show that quivers associated to double Bruhat cells have maximal green sequences \cite{Yakimov}. Muller and Speyer show that cluser algebras from grassmanians are locally-acyclic, which implies that double Bruhat cells are locally-acyclic \cite{MS}. In the defining work of upper cluster algebras \cite{cluster3} Berenstein, Fomin, and Zelevinsky show that cluster algebras from double Bruhat cells are equal to their upper cluster algebra. Together, these results imply that Conjecture~\ref{conj:big} holds for double Bruhat cells. Bucher, Machacek, and Shapiro have provided an example of where the equivalence of the upper cluster algebra and cluster algebra depends on a choice of ground ring \cite{BMS}. Bucher and Machacek have shown that Banff quivers have green-to-red sequences \cite{BS} and Muller has shown that they are locally-acyclic \cite{muller}. Hence,  Conjecture~\ref{conj:big} for the class of Banff quivers. 

In Section~\ref{sec:background} we recall relevant background information on several topics. In Section~\ref{sec:cluster_algebra} we provide the details for the proofs of Theorem~\ref{thm:big1} and Theorem~\ref{thm:big2}. In Section~\ref{sec:counterexample} we provide the details to the proof of Theorem~\ref{thm:big3}

\section{Background}\label{sec:background}

\subsection{Quivers and mutation}
\begin{definition}
A \hdef{(cluster) quiver} is a directed multigraph with no loops or 2-cycles. 
The number of vertices is called the \hdef{rank} of $Q$. 
When a quiver has no directed cycle it is called \hdef{acyclic}.

An \hdef{ice quiver} is a pair $(Q, F)$ where $Q$ is a quiver and $F$ is a subset of the vertices of $Q$ called \hdef{frozen vertices;} such that there are no edges between frozen vertices. If a vertex of $Q$ is not frozen it is called \hdef{mutable}. 
 The \hdef{mutable part} of an ice quiver $(Q,F)$ is the quiver obtained by deleting all of the vertices of $F$ and arrows with an endpoint in $F$.  An ice quiver is said to be acyclic if its mutable part is acyclic. The rank of an ice quiver is the rank of its mutable part. 
\end{definition}

We now define a procedure called mutation that takes a quiver and a vertex to produce a new quiver.

\begin{definition}\label{def:mutation}
Let $Q$ be a quiver and $k$ a vertex of $Q$. The \hdef{mutation} of $Q$ at vertex $k$ is denoted by $\mu_k$, and is a transformation of $Q$ to a new quiver $\mu_k(Q)$ that has the same vertex set, but making the following adjustment to the arrows:  \begin{enumerate}
\item For every 2-path $i \rightarrow k \rightarrow j$, add a new arrow $i \rightarrow j$. 
\item Reverse the direction of all arrows incident to $k$. 
\item Delete a maximal collection of 2-cycles created during the first step.
\end{enumerate} 
The mutation of an ice quiver follows the same process as the mutation of a quiver except that we only allow mutation to occur at mutable vertices and in the third step of the procedure we delete a maximal collection of 2-cycles and any arrows that are created between frozen vertices. 
\end{definition}

Mutation is an involution and gives an equivalence relation on the set of quivers. We
denote the equivalence class of all quivers under this relation by $\mut(Q)$. That is, $\mut(Q)$ is the set of all quivers that can be obtained from $Q$ by a finite sequence of mutations. 
If $\mut(Q)$ is finite (resp., infinite), then $Q$ is \hdef{mutation-finite} (resp., \hdef{mutation-infinite}). Mutation finite quivers have been classified by Felikson, Shapiro, and Tumarkin. 

\begin{theorem}{\cite{FST}}\label{thm:mufinte}
A quiver $Q$ is mutation finite if and only if \begin{enumerate}
\item $Q$ has rank 2;
\item $Q$ arises from a triangulation of a marked surface;
\item or $Q$ is in one of 11 exceptional mutation classes. 
\end{enumerate}
\end{theorem}

We denote the vertices of a quiver $Q$ by $Q_0$ and the arrows by $Q_1$. 
\begin{definition}{\cite[Definition 2.4]{brustle}}
The \hdef{framed quiver} associated to a quiver $Q$ is the ice quiver $(\hat{Q},Q_0^*)$ such that:

$$Q_0^*=\{i^*\text{ }|\text{ }i\in Q_0\}, \hspace{.6cm} \hat{Q}_0 = Q_0 \sqcup Q_0^*$$
$$\hat{Q}_1 = Q_1 \sqcup \{i \to i^*\text{ }|\text{ }i \in Q_0\}$$
Since the frozen vertices of the framed quiver are so natural we will simplify the notation and just write $\hat{Q}$. 
\end{definition}
\begin{definition}{\cite[Definition 2.5]{brustle}}\label{def:green1}
Let $R \in \mut(\hat{Q})$. \\A mutable vertex $i \in R_0$ is called \hdef{green} if $$\{j^*\in Q_0^*\text{ }| \text{ } \exists \text{ } j^* \rightarrow i \in R_1 \}=\emptyset.$$ It is called \hdef{red} if $$\{j^*\in Q_0^*\text{ }| \text{ } \exists \text{ } j^* \leftarrow i \in R_1 \}=\emptyset.$$ 
\end{definition}

While it is not clear from the definition that every mutable vertex in $R_0$ must be
either red or green, this was shown to be true for quivers in~\cite{derksen} and
was also shown to be true in a more general setting in~\cite{gross}.

\begin{theorem}\cite{derksen,gross}\label{thm:signcoh}
Let $R \in \mut(\hat{Q})$. Then every mutable vertex in $R_0$ is either red or green.
\end{theorem}

\begin{definition}[{\cite[Definition 2.8]{brustle}, \cite[Section 5.14]{keller2}}]
A \hdef{green sequence} for $Q$ is a sequence of vertices $\mgsi=(i_1, \dotsc, i_l)$ of $Q$ if for any $1\leq k \leq l$, the vertex $i_k$ is green in $\mu_{i_{k-1}}\circ \dotsb \circ \mu_{i_1}(\hat{Q})$.
A green sequence $\mgsi$ is \hdef{maximal} if every mutable vertex in $\mu_{i_{l}}\circ \dotsb \circ \mu_{i_1}(\hat{Q})$ is red. We will usually denote the composition $\mu_{i_{l}}\circ \dotsb \circ \mu_{i_1}$ by $\mu_{\mgsi}$. 
A \hdef{green-to-red sequence} is a mutation sequence $\mu$ such that all the vertices of $\mu(\hat{Q})$ are red. 
\end{definition} 


We recall the result of the authors previous work that classifies the existence of maximal green sequences for mutation-finite quivers \cite{mills}.
{

\begin{theorem}\cite{mills}\label{thm:1}
For every finite mutation class of quivers, either there is no quiver in the mutation class with a maximal green sequence or every quiver in the mutation class has a maximal green sequence. Furthermore, the only mutation class whose quivers have no maximal green sequence are $\mathbb X_7$ and those arising from once-punctured closed surfaces. 
\end{theorem} 

}


\subsection{Cluster algebras with principal coefficients}\label{sec:cluster_algebra}
%

Let $\tilde{Q} = (Q,F)$ be an ice quiver with mutable vertices $\{1,\dots,n\}$ and frozen vertices $\{1^*,\dots, j^*\}$. Let $\mathcal{Z}=\mathcal{Z}(\tilde{Q}) =  \Z[y_1^{\pm 1},y_2^{\pm 1},\dots,y_j^{\pm 1}]$ and take $\O$ to be any open subset of $\spec(\mathcal{Z})$. 
We take the ground ring for our cluster algebra to be the ring of regular functions on $\O$. 
In this work we will only consider open sets of the form $U = \bigcap_{f \in \Lambda} D_f$ for some finite set $\Lambda \subset \ZZ$. Here $D_f$ denotes the Zariski open set $D_f = \{ p \in \spec R | f \not\in p\}.$ In this case the ring of regular functions is $\ZZ[\Lambda^{-1}]$ where $\Lambda^{-1} = \{ f^{-1} | f \in \Lambda\}$ which we will denote as $\O_Y.$

We now extend the operation of mutation to construct the generators for our algebra.

Let $\F$ be a field that contains $\O_Y$. A \hdef{seed} in $\F$ is a tuple $(\x,\tilde R),$ where 
\begin{itemize}
\item the \hdef{cluster}: $\x = (x_1,\dots,x_n)$ is an $n$-tuple of elements of $\F$ which freely generate $\F$ as a field over the fraction field of $\O_Y$, and
\item $\tilde R =(R,F)$ is an ice quiver in $\mut(\tilde{Q}).$
\end{itemize}

\begin{definition}[Seed mutation.]
Let $((x_1,\dots,x_n), \tilde{R})$ be a seed in $\F$, and let $1 \leq k \leq n$. The \hdef{seed mutation} $\mu_k$ in direction $k$ transforms $(\x, \tilde{R})$ into the seed $(\x', \tilde{R'})$ which is defined by $\x' := (x_1',\dots, x_n')$ where $x_j' = x_j$ if $j \neq k$, and $$x_k' = \frac{\displaystyle \left( \prod_{k \rightarrow j^*}y_j\right) \left(\prod_{k \rightarrow j} x_j \right)+ \left(\prod_{i^* \rightarrow k} y_i \right) \left(\prod_{i \rightarrow k} x_i \right)}{x_k}.$$
Here the products run over all arrows of $\tilde{R}$ and
 $\tilde{R'} = (\mu_k(R),F).$
\end{definition}

%

\begin{definition}
Consider the undirected $n-$regular tree $\T_n$ whose edges are labeled by the numbers $1,\dots,n$, so that each of the $n$ edges adjacent to each vertex all have different labels. A \hdef{cluster pattern} is an assignment of a seed $(\x_t, \tilde Q_t)$ to every vertex $t \in \T_n$, such that the seeds assigned to the endpoints of any edge $t - t'$ labeled $k$ are obtained from each other by the seed mutation in direction $k.$ Cleary, a cluster pattern is uniquely determined by an arbitrary seed. 
\end{definition}

\begin{definition}
Given a cluster pattern we denote $$\mathcal{X} = \bigcup_{t \in \T_n} \x_t.$$ That is $\mathcal{X}$ is the union of clusters of all the seeds in the pattern. We denote the elements in the cluster as $$\x_t = (x_{1;t},\dots, x_{n;t}).$$ The elements $x_{j;t}$ are called \hdef{cluster variables}. For a cluster pattern containing the seed $(\x, \tilde Q)$ and a choice of open set $\O \subset \spec R$ the \hdef{cluster algebra} $\A_\O(\tilde Q)$ is the $\O_Y$-subalgebra of the ambient field $\F$ generated by all cluster variables, that is $\mathcal{A}_\O(\tilde Q) := \O_Y[\mathcal{X}]$. 
\end{definition}

\begin{theorem}\cite[Theorem 3.1]{cluster1}\label{thm:Laurent}
The cluster algebra $\A_\O$ associated with a seed $(\x,\tilde Q)$ is contained in the Laurent polynomial ring $\O_Y[\x^{\pm 1}]$, i.e., every element of $\A_\O$ is a Laurent polynomial over $\O_Y$ in the cluster variables from $\x$. 
\end{theorem}

\begin{theorem}\cite{MSW,leeschiffler}\label{thm:Positivity}
For any cluster algebra $\A_\O$, and any seed $(\x,\tilde R)$, and any cluster variable $x \in \A_\O$, the Laurent polynomial for $x$ expressed in terms of the cluster variables of $\x$ has coefficients that are non-negative integer linear combinations of elements in $\{y_1,\dots,y_j\}$. 
\end{theorem}

In \cite{cluster4}, Fomin and Zelevinsky introduced a special type of coefficients that are called principal coefficients.

\begin{definition}
We say that a cluster algebra $\A_\O(\tilde Q)$ has \hdef{principal coefficients} if there exists some ice quiver $(R,F) \in \mut(\tilde Q)$ such that $(R,F)$ is isomorphic to the framed quiver of the mutable part of $R$. In this case, we denote $\A_\O(\tilde Q)$ as $\A_\O^{\bullet}(R)$.
\end{definition}

In most cases there are infinitely many cluster variables generated by mutation. 
Sometimes it is more convenient to use a matrix to encode the combinatorial data of mutation in place of a quiver. 

\begin{definition}
Let $(Q,F)$ be an ice quiver with mutable vertices $\{1,\dots,n\}$ and frozen vertices $\{n+1,\dots,m\}$. We associate an $m \times n$ matrix $\tilde B{(Q,F)}=(b_{ij})$, where $b_{ij} = 0$ if $i=j$ and otherwise $b_{ij}$ is the number of arrows $j \rightarrow i \in Q_1$ for appropriate $i$ and $j$.  Note that if $j \rightarrow i \in Q_1$ then we adopt the convention that this is a ``negative arrow'' and $b_{ji} < 0$. We call the matrix $\tilde B$ the \hdef{exchange matrix} and the top $n \times n$ matrix the \hdef{principle part} of the matrix and denote it by $B$. Note that by construction the principal part of any exchange matrix will always be skew-symmetric. We sometimes use $(\x, \tilde B(Q,F))$ or just $(\x, \tilde B)$ to denote the seed $(\x , (Q,F)).$
\end{definition}

\subsection{Upper cluster algebras}
\begin{definition}[Upper cluster algebra]
The upper cluster algebra $\U_\O$ of $\A_\O$ is the intersection of a Laurent ring corresponding to each cluster of $\A_\O$. That is $$\U_\O := \bigcap_{t \in \T_n} \O_Y[x_{1;t}^{\pm 1},\dots,x_{n;t}^{\pm 1}]$$
\end{definition}

It follows immediately from Theorem~\ref{thm:Laurent} that every cluster algebra is naturally contained in its upper cluster algebra. As most cluster algebras have infinitely many seeds it is very hard to provide a presentation for an upper cluster algebra. Muller and Matherne \cite{MM} provided an algorithm that will compute the presentation of an upper cluster algebra, but requires the upper cluster algebra to be finitely generated and totally coprime to guarantee it succeeds. 

\begin{definition}
A seed $(\mathbf x, \tilde B)$ or matrix $\tilde B$ is called \hdef{coprime} if every pair of columns of $\tilde B$ are linearly independent. A cluster algebra is \hdef{totally coprime} if every seed is coprime.
\end{definition}

When a cluster algebra is totally coprime we can simplify our definition of the upper cluster algebra. Specifically, we only need to check that a polynomial is Laurent in the initial cluster and those that are exactly one mutation away as opposed to checking every cluster.

\begin{theorem}\cite[Lemma]{cluster3}\label{thm:fullrank}
If the exchange matrix $B_Q$ is totally coprime then $$ \U_\O(Q) =\O_Y[x_1^{\pm 1}, \dots, x_n^{\pm 1}] \cap \left( \bigcap_{i\in \{1,\dots,n\}} \O_Y[x_1^{\pm 1}, \dots,(x_i')^{\pm 1} ,\dots x_n^{\pm 1}]\right)$$ where $x_i'$ is the cluster variable obtained from mutating the initial seed at vertex $i$. 
\end{theorem}

\subsection{Locally-acyclic cluster algebras}\label{sec:laca}

We now introduce Muller's locally-acyclic cluster algebras.

\subsubsection{Cluster localizations and covers}

\begin{definition}
Let $(\x,\tilde Q) = ((x_1,\dots,x_n),(Q,F))$ be a seed with $F = \{1^*,\dots,j^*\}$. A \hdef{freezing} of a seed at the cluster variables $x_{i_1}, \dots x_{i_k}$ is the new seed $(\x^\dagger,\tilde Q^\dagger)$ where $\x^\dagger = (x_1,\dots, x_{i_1-1},x_{i_1+1},\dots, x_{i_2-1},x_{i_2+1},\dots,x_n)$ is $\x$ with each of the $i_k$-th entries removed and $\tilde Q^\dagger = (Q,F\cup \{{i_1},\dots,{i_k}\})$. For an open set $\bigcap_{f \in \Lambda} D_f = \O \in \spec \Z[y_1,\dots,y_j]$ we define $\O^\dagger = \bigcap_{f \in \Lambda'} D_{f} \subset \spec \Z[y_1,\dots,y_j,x_{i_1},\dots,x_{i_k}]$ where $\Lambda' = \{x_{i_1},\dots, x_{i_k}\} \cup \Lambda$ and $f \in \Lambda$ is identified with its natural image in $ \Z[y_1,\dots,y_j,x_{i_1},\dots,x_{i_k}].$
\end{definition}

Note there is a natural inclusion of rings $\O_Y \hookrightarrow \O^\dagger_Y$ and furthermore 
$\O_Y^\dagger \cong \O_Y[x_{i_1}^{\pm 1},x_{i_2}^{\pm 1},\dots, x_{i_k}^{\pm 1}]$ as rings. The following lemma of Muller was proven in a slightly different setting, but the proof carries over to our setting with little work. 

\begin{lemma}\cite[Lemma 1]{muller}\label{lem:cluster_localization}
Let $\{x_{i_1},x_{i_2},\dots,x_{i_k}\} \subset \x$ be a set of cluster variables in a seed $(\x,\tilde Q),$ with freezing $(\x^\dagger, \tilde Q ^ \dagger)$. Let $\A_\O$ and $\U_\O$ be the cluster algebra and the upper cluster algebra of $(\x,\tilde Q)$ and let $\A_\O^\dagger$ and $U_\O^\dagger$ be the cluster algebra and upper cluster algebra of $(\x^\dagger, \tilde Q^\dagger)$ generated over the ground ring $\O_Y^\dagger$. Then there are inclusions in $\F$ $$\A_\O^\dagger \subseteq \A_\O[(x_{i_1}x_{i_2}\dots x_{i_k})^{-1}] \subset \U_\O[(x_{i_1}x_{i_2}\dots x_{i_k})^{-1}] \subset \U_\O^\dagger.$$
\end{lemma}

\begin{definition}
If $(\x^\dagger, \tilde Q^\dagger)$ is a freezing of $(\x,\tilde Q)$ at $\{x_{i_1},x_{i_2},\dots,x_{i_k}\}$ such that $$\A_\O^\dagger = \A_\O[(x_{i_1}x_{i_2}\dots x_{i_k})^{-1}] ,$$ then we say that $\A_\O^\dagger$ is a \hdef{cluster localization} of $\A_\O.$
\end{definition}

Note that if $\A_\O^\dagger = \U_\O^\dagger$  then it follows immediately from Lemma~\ref{lem:cluster_localization} that the freezing is a cluster localization. Also, cluster localizations are transitive. That is if $\A^\dagger$ is a cluster localization of $\A$, and $\A^\ddagger$ is a cluster localization of $\A^\dagger$, then $\A^\ddagger$ is a cluster localization of $\A$. 

\begin{definition}
For a cluster algebra $\A_\O$, a set of cluster localizations $\{ \A_i\}$ of $\A_\O$ is called a \hdef{cover} if, for every proper prime ideal $P \subsetneq \A_\O$ there exists some $\A_i$ in the set such that $\A_iP \subsetneq \A_i$. 
\end{definition}

Like cluster localizations covers are transitive. If $\{\A_i\}$ is a cover of $\A_\O$, and $\{\A_{ij}\}$ is a cover of $\A_i$, then $\bigcup \{ \A_{ij}\}$ is a cover of $\A$. However, there does not exist a notion of a common refinement of two covers for a cluster algebra. 

%

The main argument used in this paper to construct covers is summarized in the following lemma. It is the idea behind the Banff algorithm from \cite{muller} that was introduced to explicitly construct covers. We say that two cluster variables $x_1$ and $x_2$ are \hdef{coprime} if the ideal generated by $x_1$ and $x_2$ is the entire cluster algebra $\A_\O.$

\begin{lemma}\label{lem:coprimecover}
If $x_1$ and $x_2$ are coprime cluster variables and the freezings $\A_1 = \A_\O[x_1^{- 1}]$ and $\A_2 = \A_\O[x_2^{- 1}]$ are cluster localizations then the set $\{\A_1, \A_2\}$ is a cover for $\A_\O$. 
\end{lemma}
\begin{proof}
Any proper prime ideal of $\A_\O$ cannot contain both $x_1$ and $x_2$. 
\end{proof}
\subsubsection{Locally-acyclic cluster algebras}

\begin{definition}{\cite{muller}}
A cluster algebra is said to be locally-acyclic if it admits a cover by acyclic cluster algebras. 
\end{definition}

As we mentioned before the notion of locally-acyclic cluster algebras generalizes the notion of acyclic cluster algebras so naturally acyclic cluster algebras should be locally-acyclic. 

\begin{theorem}\cite{muller}\label{thm:acyclic_LA}
If a cluster algebra is acyclic then it is locally-acyclic. 
\end{theorem}

The main result we would like to use in this work is that for locally-acyclic cluster algebras we have $\A_\O =\U_\O$. 

\begin{theorem}\cite[Theorem 2]{MullerAU}\label{thm:LA_AU}
If $\mathcal{A}_\O$ is locally-acyclic then $\mathcal{A}_\O=\mathcal{U}_\O$. 
\end{theorem}

Provided that we know that the freezings of a cluster algebra $\A$ are locally-acyclic we can show that $\A$ is locally-acyclic.

\begin{lemma}\label{lem:short_LA}
If there exists a seed $(\x,\tilde{Q})$ in a cluster algebra $\A_\O$ with coprime cluster variables $x_i,x_j \in \x$ such that the cluster algebra $\A^\dagger$ obtained from the freezing at $x_i$ and the cluster algebra $\A^\ddagger$ obtained from the freezing at $x_j$ are both locally-acyclic then $\A_\O$ is locally-acyclic. 
\end{lemma}
\begin{proof}
If $\A^\dagger$ is locally-acyclic we have that $\A_\O^\dagger = \U^\dagger$ by Theorem~\ref{thm:LA_AU}. Therefore it is a cluster localization by Lemma~\ref{lem:cluster_localization}. Similarly, $\A_\O^\ddagger$ is a cluster localization. By Lemma~\ref{lem:coprimecover} we have that $\A^\dagger_\O$ and $\A^\ddagger_\O$ form a cover for $\A_\O$. 

Now $\A_\O$ is covered by the locally-acyclic cluster algebras. Then there exists acyclic covers $\{\A^\dagger_{j}\}$ and $\{\A^\ddagger_{k}\}$ for $A^\dagger_\O$ and $A_\O^\ddagger$, respectively. Using the transitivity of covers we know that the union $\{\A^\dagger_{i}\} \cup \{\A^\ddagger_{j}\}$ is an acyclic cover of $\A_\O$. Therefore $\A_\O$ is locally-acyclic. 
\end{proof}

Muller's initial work \cite{muller} found a sufficient condition on the quiver of a seed to determine if two cluster variables are coprime for any choice of ground ring. A \hdef{bi-infinite path} in an ice quiver $(Q,F)$ is a sequence of mutable vertices $\{i_k\}_{k \in \Z}$ such that $i_k \rightarrow i_{k+1} \in \tilde Q_1$ for all $k \in \Z.$ An arrow $i \rightarrow
j$ is in a bi-infinite path $\{i_k\}_{k\in Z}$ if there exists a $k^* \in \Z$ such that $i=i_{k^*}$ and $j = i_{k^*+1}.$ Note that an arrow does not need to be a part of a cycle for it to be in a bi-infinite path. 

\begin{lemma}\cite[Proposition 5.1]{muller} 
An arrow $e$ is in a bi-infinite path if and only if there is a cycle upstream from $e$ and a cycle downstream from $e$. 
\end{lemma}

Now we say that a pair of cluster variables $x_i$ and $x_j$ are a \hdef{covering pair} for a cluster algebra $\A_Q$ if there exists a seed $(\x,\tilde Q)$ such that $x_i, x_j \in \x$ and there exists an arrow $i \rightarrow j \in \tilde{Q}_1$ that is not contained in any bi-infinite path of $\tilde Q$. 
\begin{lemma} \cite[Lemma 5.3]{muller}\label{lem:cover_is_coprime} The cluster variables of a covering pair are coprime. 
\end{lemma}
Note that any arrow with an endpoint at a sink or source of a quiver is not in any bi-infinite path. 

\subsection{Cluster algebras from surfaces}\label{sec:surfaces}
The connection between marked surfaces and cluster algebras was first established by Fomin, Shapiro, and Thurston \cite{fomin}.

Let $S$ be an orientable 2-dimensional Riemann surface with or without boundary. We designate a finite number of points, $M$, in the closure of $S$ as marked points. We require at least one marked point on each boundary component. We call marked points in the interior of $S$ \textbf{punctures.} Together the pair $\Sigma=(S,M)$ is called a \textbf{marked surface.} 
For technical reasons we exclude the cases when $\Sigma$ is one of the following: 
\begin{itemize}
\item a sphere with less than four punctures;
\item an unpunctured or once punctured monogon;
\item an unpunctured digon; or 
\item an unpunctured triangle. 
\end{itemize}
Note that the construction allows for spheres with four or more punctures.

Up to homeomorphism a marked surface is determined by four things. The first is the genus $g$ of the surface. The second is the number of boundary components $b$ of $S$. The third is the number of punctures $p$ in $M$, and the fourth is the set $m=\{m_i\}_{i=1}^{b}$ where $m_i\in \mathbb{Z}_{>0}$ denotes the number of marked points on the $i$th boundary component of $S.$ We say a marked surface is \textbf{closed} if it has no boundary. 
\begin{definition}
An \textbf{arc} $\gamma$ in $(S,M)$ is a curve in $S$ such that:
\begin{itemize}
\item The endpoints of $\gamma$ are in $M$. 
\item $\gamma$ does not intersect itself, except that its endpoints may coincide.
\item $\gamma$ is disjoint from $M$ and the boundary of $S$, except at its endpoints. 
\item $\gamma$ is not isotopic to the boundary, or the identity. 
\end{itemize}
An arc is called a \textbf{loop} if its two endpoints coincide.

Each arc is considered up to isotopy. Two arcs are called compatible if there exists two arcs in their respective isotopy classes that do not intersect in the interior of $S$. 
\end{definition}

\begin{definition}
A \textbf{taggd arc} is constructed by taking an arc that does not cut out a once-punctured monogon and marking or "tagging" its ends as either \textbf{plain} or \textbf{notched} so that: \begin{itemize}
\item an endpoint lying on the boundary of $S$ is tagged plain; and
\item both ends of a loop must be tagged in the same way. 
\end{itemize} 
In figures we use a $\bowtie$ to denote when the tagging of an arc is notched. Two tagged arcs are considered compatible if: \begin{itemize}
\item Their underlying untagged arcs are the same, and their tagging agrees on exactly one endpoint. 
\item Their underlying untagged arcs are distinct and compatible, and any shared endpoints have the same tagging. 
\end{itemize}
 A maximal collection of pairwise compatible tagged arcs is called a \textbf{(tagged) triangulation} of $(S,M)$. 
\end{definition}

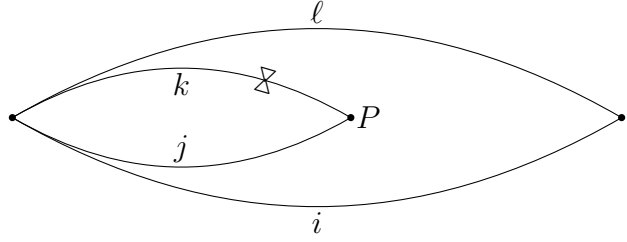
\begin{figure}[t]
\begin{center}
\begin{tikzpicture}[scale = 0.9]
\draw (0,0) to[bend left] node[above]{$\ell$} (9,0) to[bend left] node [below]{$i$} (0,0);
\draw (0,0) to[bend left]  node [below]{$k$} node[near end,sloped,rotate=90]{$\bowtie$} (5,0) to[bend left] node [above]{$j$} (0,0);

\fill (0,0) circle (1.5pt);
\fill (5,0) circle (1.5pt);
\fill (9,0) circle (1.5pt);
\draw (5,0) node[right] {$P$};
\end{tikzpicture}
\end{center}
\caption[Example of a radial puncture.]{The puncture $P$ is a radial puncture.}\label{fig:radial}
\end{figure}

\begin{definition}
We call a puncture $P$ a \textbf{radial puncture in a tagged triangulation} if and only if $P$ is the unique puncture in the interior of a digon and there exists two arcs in the interior of this digon that differ only by their tagging at $P$. See Figure~\ref{fig:radial}.
\end{definition}

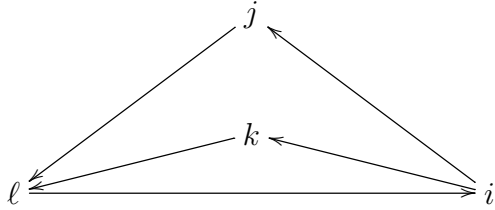
\begin{figure}[t]

\centering
$$
\begin{xy} 0;<-.9pt,0pt>:<0pt,-.9pt>:: 
(0,75) *+{i} ="0",
(100,0) *+{j} ="1",
(100,50) *+{k} ="2",
(200,75) *+{\ell} ="3",
"0", {\ar"1"},
"0", {\ar"2"},
"1", {\ar"3"},
"2", {\ar"3"},
"3",{\ar"0"},
\end{xy}$$
\caption[Quiver corresponding to a punctured digon.]{A special situation in the construction of a quiver from the triangulation given in Figure~\ref{fig:radial}.}\label{fig:arc}
\end{figure}

\begin{definition}\label{def:quivfromtri}
Let $T$ be a triangulation of a marked surface. The \textbf{quiver associated to $T$}, which we will denote as $Q_T$, is the quiver obtained from the following construction. For each arc $\alpha$ in a triangulation $T$ add a vertex $v_\alpha$ to $Q_T$. If $\alpha_i$ and $\alpha_j$ are two edges of a triangle in $T$ with $\alpha_j$ following $\alpha_i$ in a clockwise order, then add an edge to $Q_T$ from $v_{\alpha_i} \rightarrow v_{\alpha_j}$. If $\alpha_k$ and $\alpha_j$ have the same underlying untagged arc as in Figure~\ref{fig:radial} we refer you to Figure~\ref{fig:arc} for the construction in this situation. Note that the quiver is the same whether $\alpha_j$ or $\alpha_k$ is tagged. More generally distinct triangulations may yield the same quiver. 
\end{definition}

Finally we construct a cluster algebra from a surface by considering the the cluster algebra generated from the quiver associated to a triangulation.

\begin{definition}
For any surface $\Sigma$ with triangulation $T$ the cluster algebra associated to the marked surface $\Sigma$ is the cluster algebra $\A_\O(\tilde{Q}_T).$
\end{definition}

In this analogy where surfaces correspond to cluster algebras we have that arcs correspond to cluster variables and triangulations correspond to clusters. Musiker, Schiffler, and Williams give a combinatorial formula for the Laurent expression of any arc in a tagged triangulation \cite{MSW}. In \cite{MSW2} they extend the construction to give the Laurent expression for any non-contractible loop of the surface. For a fixed triangulation $T$ of a marked surface $\Sigma$ we denote the Laurent expression in $\A_\O^\bullet(Q_T)$ of an arc or non-contractible loop $\gamma \in \Sigma$ by $x_\gamma.$ 

It is possible and oftentimes easier to find algebraic relations between cluster variables by working with the arcs of the surface.

\begin{definition}\cite{MSW2}\label{def:smoothing}
Let $\alpha$, $\gamma_1$ and $\gamma_2$ be arcs and closed loops such that \begin{itemize}
\item $\gamma_1$ crosses $\alpha_2$ at a point $x$. 
\item $\gamma$ has a self crossing at a point $x$. 
\end{itemize}
We take $C = \{\gamma_1,\gamma_2\}$ (respectively, $C=\{\gamma\}$) and define the smoothing of $C$ at $x$ to be the pair of multicurves $C_+ = \{\alpha_1,\alpha_2\}$ (respectively, $\{\alpha\}$) and $C_-=\{\beta_1,\beta_2\}$ (respectively, $\{\beta\}$). Here  $C_+$ is obtained by replacing the crossing $\times$ with $ \ou $ and similarly $C_-$ is obtained by replacing the crossing $\times$ with $ \supset\subset$. See Figure~\ref{fig:smoothing}.
\end{definition}

\begin{figure}
\includegraphics[scale=.7]{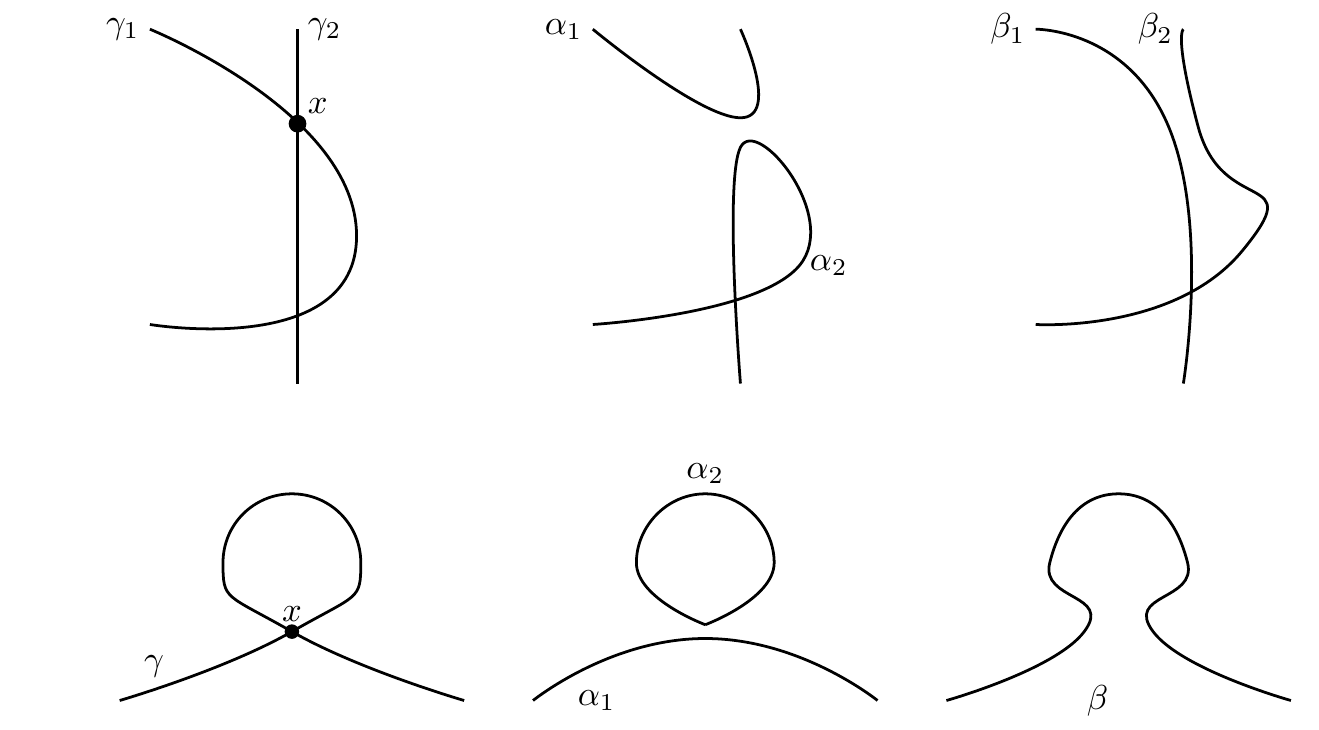}
\caption[The smoothing of crossings]{The two cases for smoothing a crossing as defined in Definition~\ref{def:smoothing}.}\label{fig:smoothing}
\end{figure}

Then a smoothing gives rise to the following algebraic relations that we use in Section~\ref{sec:LA_finite} to find coprime cluster variables in cluster algebras arising from marked surfaces. 

\begin{theorem}\cite{MW}\label{thm:smooth_relation}
Let $C,$ $C_+$, and $C_-$ as in the previous definition. Then we have the relation $ x_C = Y_1 x_{C+} +Y_2x_{C_-}$ in $\A_\O^\bullet(Q_T)$ where $Y_1$ and $Y_2$ are monomials in the coefficient variables $y_{\tau_i}$
\end{theorem}

\section{Proof of Theorem~\ref{thm:big1} and Theorem~\ref{thm:big2}}\label{chapter:LA}

\subsection{Exceptional mutation-finite quivers}

Muller showed that the extended affine mutation classes $E_6^{(1,1)},E_7^{(1,1)},$ and $E_8^{(1,1)}$ are locally-acyclic \cite{muller}. It is known that the cluster algebra and upper cluster algebra are distinct for $\mathbb X_7$, but to the authors knowledge this fact was never written down. 
\begin{theorem}\label{thm:LA_exceptional}
All of the exceptional mutation-finite mutation classes generate locally-acyclic cluster algebras except for $\mathbb{X}_7.$ Furthermore, the cluster algebra for $\mathbb{X}_7$ is not equal to its upper cluster algebra.
\end{theorem} 
\begin{proof}
The exceptional mutation classes $E_6, E_7, E_8, \tilde{E}_6, \tilde E_7$ and $\tilde{E}_8$ are all acyclic so they are locally-acyclic by Theorem~\ref{thm:acyclic_LA}. Muller showed that $E_6^{(1,1)},E_7^{(1,1)},E_8^{(1,1)},$ and $\mathbb X_6$ are locally-acyclic \cite{muller}. 

It remains to be shown that the cluster algebra for $\mathbb{X}_7$ is not locally-acyclic. We first show that $\A_\O^\bullet \neq \U_\O$ by a standard grading argument. 

Note that there are two quivers in the mutation class of $\mathbb X_7$ and their corresponding exchange matrices are $$B_1 = \left(\begin{array}{rrrrrrr}
0 & 1 & -1 & 1 & -1 & 1 & -1 \\
-1 & 0 & 2 & 0 & 0 & 0 & 0 \\
1 & -2 & 0 & 0 & 0 & 0 & 0 \\
-1 & 0 & 0 & 0 & 2 & 0 & 0 \\
1 & 0 & 0 & -2 & 0 & 0 & 0 \\
-1 & 0 & 0 & 0 & 0 & 0 & 2 \\
1 & 0 & 0 & 0 & 0 & -2 & 0
\end{array}\right)
\text{ and } B_2 = \left(\begin{array}{rrrrrrr}
0 & -1 & 1 & -1 & 1 & -1 & 1 \\
1 & 0 & 1 & 0 & -1 & 0 & -1 \\
-1 & -1 & 0 & 1 & 0 & 1 & 0 \\
1 & 0 & -1 & 0 & 1 & 0 & -1 \\
-1 & 1 & 0 & -1 & 0 & 1 & 0 \\
1 & 0 & -1 & 0 & -1 & 0 & 1 \\
-1 & 1 & 0 & 1 & 0 & -1 & 0
\end{array}\right).$$

We consider the cluster algebra $\A_\O^\bullet(B_1)$ for an arbitrary open set $\O$. 
It is clear that both of the matrices $B_1$ and $B_2$ are coprime since for each pair of columns $c_1$ and $c_2$ in one of the matrices there is a row where $c_1$ has a 0 entry and $c_2$ has a non-zero entry and vice-versa. Therefore the cluster algebra is totally coprime. It follows from Lemma~\ref{thm:fullrank} that $Z = (y_2y_3x_2^2+x_3^2+y_2x_1)/(x_1x_2)$ is in the upper cluster algebra. 

Furthermore by \cite[Proposition 3.2]{Grabowski} taking $\deg(x_1) = 2$ and $\deg(x_i)=1$ for $2\leq i \leq 7$ the cluster algebra is graded with its generators concentrated in degrees 1 and 2. Note that the grading imposed on the cluster algebra is independent of our choice for $\O$.
Since $\deg(Z)=0$ we have that $\A_\O^\bullet(B_1) \neq \U_\O$. Then it follows that it is not locally-acyclic by Theorem~\ref{thm:LA_AU}.
\end{proof}

\subsection{Cluster algebras from surfaces}\label{sec:LA_finite}

We continue our discussion with known results about the local-acyclicity of surface cluster algebras. Recall that marked surfaces are parameterized by four values: the genus of the surface, the number of boundary components, the total number of marked points, and the number of marked points on each boundary component. Note every boundary component must have at least one marked point.

\begin{lemma}\cite{ladkani}\label{lem:1punc_no_au}
For a once-punctured surface with no boundary components the upper cluster algebra is not equal to the cluster algebra. 
\end{lemma}
\begin{corollary}
Once-punctured closed surfaces are not locally-acyclic.  
\end{corollary}
\begin{proof}
The claim follows immediately from Lemma~\ref{lem:1punc_no_au} and Theorem~\ref{thm:LA_AU}. 
\end{proof}

Note that these are the only surface-type cluster algebras whose quivers do not admit a maximal green sequence. 

For surface cluster algebras with no punctures it has been shown that the bangle and bracelet collections are bases for the cluster algebra \cite{cls,MSW2}. Due to work of Fock and Goncharov \cite{FockGoncharov2006} this implies that $\A= \U$ for these cluster algebras.
\begin{theorem}[\cite{cls,MSW2}] \label{thm:surface_AU}
The cluster algebra associated to any surface with non-empty boundary and at least one marked point is equal to its upper cluster algebra. 
\end{theorem}

However, this work does not show that these cluster algebras are locally-acyclic. Muller showed that unpunctured surfaces with at least two marked points are locally-acyclic in \cite{muller} and Canakci, Lee, and Schiffler showed the result for surfaces with one boundary component and one puncture \cite{cls}. 

\begin{lemma}[\cite{cls,muller}]\label{thm:boundaryLA}
Any surface with non-empty boundary and at least two marked points is locally-acyclic. 
\end{lemma}

It remains to be shown that closed surfaces with at least two punctures and surfaces with exactly one boundary component and one marked point are locally-acyclic.

We begin with the closed surface case. As a matter of notation if $\gamma$ is an arc on the surface that is tagged plain at a puncture $p$ then we will denote the arc that is identical to $\gamma$ except whose tagging at $p$ is notched by $\gamma^{(p)}.$ Similarly, if $\gamma$ has endpoints at two punctues $p$ and $q$ we use $\gamma^{(pq)}$ to denote the arc that is notched at both $p$ and $q$. 

\begin{theorem}\cite[Theorem 12.10]{MSW}\label{thm:msw_formula}
Fix a tagged triangulation $T$ of a marked surface $\Sigma$.
Let $p$ and $q$ be punctures in $\Sigma$, and let $\rho$ be an arc with both taggings plain whose endpoints are $p$ and $q$.  Then we have
 $$x_\rho x_{\rho^{(pq)}}y_{\rho^{(p)}}^{\chi(\rho^{(p)} \in T)}y_{\rho^{(q)}}^{\chi(\rho^{(q)} \in T)} - x_{\rho^{(p)}}x_{\rho^{(q)}}y_{\rho}^{\chi(\rho \in T)}y_{\rho^{(pq)}}^{\chi(\rho^{(pq)} \in T)}$$ is equal to $$Y(\rho,T):= \prod_{\tau \in T} y_\tau^{e(\tau,\rho)} \left( \prod_{\tau \in T} y_\tau^{e_p^{\bowtie} (\tau)} - \prod_{\tau \in T} y_\tau^{e_p(\tau)}  \right) \left( \prod_{\tau \in T} y_\tau^{e_q^{\bowtie}(\tau)}  - \prod_{\tau \in T} y_\tau^{e_q(\tau)} \right)$$ in the cluster algebra $\A_\O^\bullet(Q_T)$ for any open set $\O \subset \spec(\ZZ(\hat Q_T))$. Here $\chi$ is 1 if the argument is true and 0 otherwise, $e(\tau,\rho)$ is the number of times $\tau$ and $\rho$ cross, and $e_p(\tau)$ (respectively, $e_p^{\bowtie}(\tau)$) is the number of ends of $\tau$ that are incident to the
puncture $p$ with a plain (respectively, notched) tagging.
\end{theorem}

Define $f(\rho,T) = \prod_{\tau \in T} y_\tau^{e(\tau,\rho)} $ and $g(T,p) =  \left( \prod_{\tau \in T} y_\tau^{e_p^{\bowtie} (\tau)} - \prod_{\tau \in T} y_\tau^{e_p(\tau)}  \right)$ so we have $Y(\rho,T) = f(\rho,T) g(T,p) g(T,q)$.

\begin{corollary} \label{cor:coprime_closed_surface}
Let $\Sigma, T, \rho, p$ and $q$ be as in Theorem~\ref{thm:msw_formula}. 
The cluster variables $x_\rho$ and $x_{\rho^{(p)}}$ are coprime in $\A_{\O}^\bullet(Q_T)$ for any open set $\O \subset \spec(\ZZ(\hat Q_T))$ that is contained in the intersection $D_{f(\rho,T)} \cap D_{g(T,p)} \cap D_{g(T,q)}$.  
\end{corollary}
\begin{proof}
 By Theorem~\ref{thm:msw_formula} we have that $Y(\rho,T)$ is in the ideal generated by $x_\rho$ and $x_{\rho^{(p)}}$. However, the polynomial $Y(\rho,T)$ is a unit in $\O_Y$ so the ideal generated by $x_\rho$ and $x_{\rho^{(p)}}$ is the entire cluster algebra.
\end{proof}

\begin{theorem}\label{thm:LA_closed_surface}
Let $\Sigma$ be closed surface with at least two punctures. 
Let $p$ and $q$ be two punctures of $\Sigma$ and let $\rho$ be the plain arc whose endpoints are $p$ and $q$. Fix a triangulation $T$ of $\Sigma$ that contains $\rho$ and $\rho^{(p)}.$
The cluster algebra $\A_{\O}^\bullet(Q_T)$ is locally-acyclic for any open set $\O \subset \spec(\ZZ(\hat Q_T))$ that is contained in the intersection $D_{f(\rho,T)} \cap D_{g(T,p)} \cap D_{g(T,q)}$.  
\end{theorem}
\begin{proof}
By Corollary~\ref{cor:coprime_closed_surface} we know that $x_\rho$ and $x_\rho^{(p)}$ are coprime and are elements of the same cluster that is associated with $T$. Let $\A^\dagger_\O$ be the cluster algebra obtained from the freezing of $x_\rho$ and $A^\ddagger_\O$ the cluster algebra obtained from the freezing of $x_{\rho^{(p)}}$. The cluster algebra $\A^\dagger_\O$ arises from the surface obtained from cutting out a disc from $\Sigma.$ In particular the cut surface will be of the same genus and have the same number of marked points as $\Sigma$, but adds a boundary component on which the marked points $p$ and $q$ now lie. An example in the genus 1 case is shown in Figure~\ref{fig:two_marked_point_torus_freeze}. Therefore $\A^\dagger_\O$ is locally-acyclic by Theorem~\ref{thm:boundaryLA}. By the symmetry of the vertices corresponding to $x_\rho$ and $x_{\rho^{(p)}}$ in the quiver $Q_T$ we see that the two freezings form isomorphic quivers so $\A^\ddagger_\O$ is also locally-acyclic. It follows from Lemma~\ref{lem:short_LA} that $\A_\O^\bullet$ is locally-acyclic. 
\end{proof}

\begin{figure}
\centering{
\underline{\includegraphics[scale=.6]{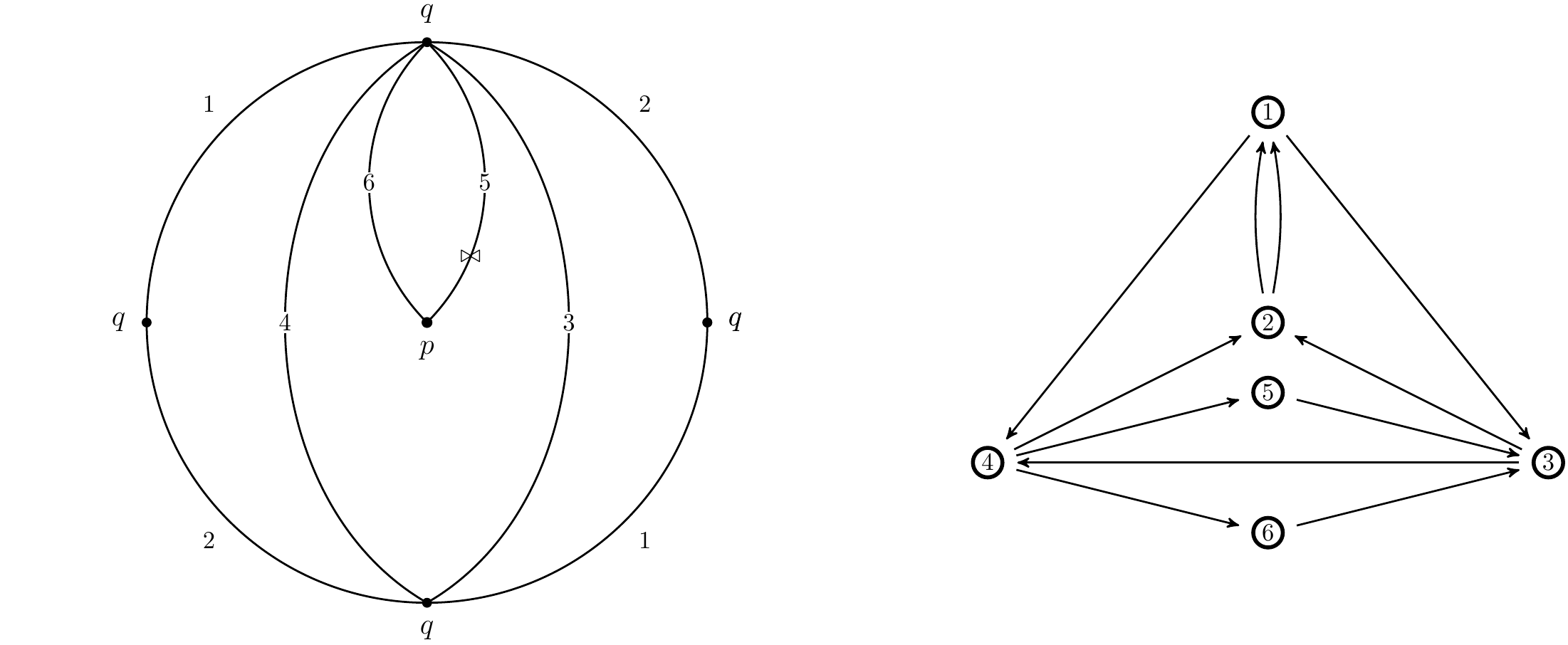}
}
\includegraphics[scale=.6]{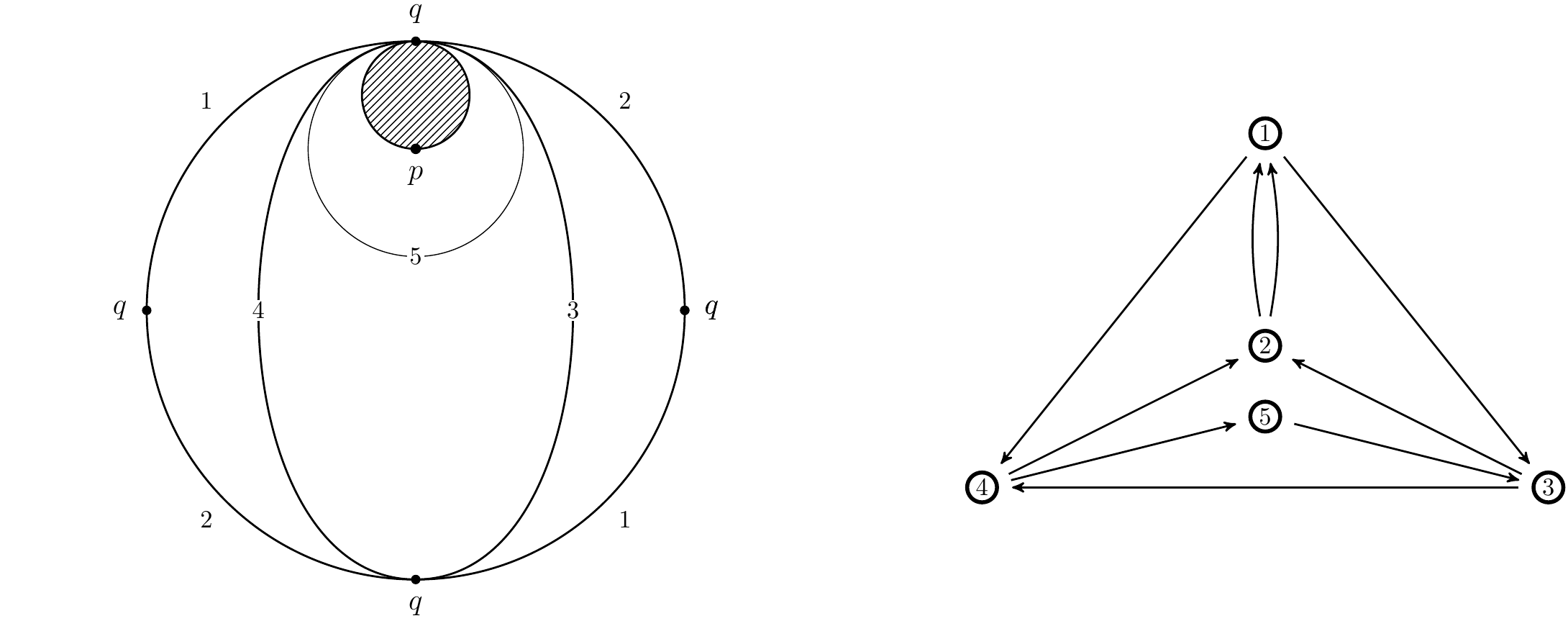}}
\caption[Freezing of a puncture-to-puncture arc]{Two quivers from marked surfaces. The freezing at vertex 6 of the top quiver produces a quiver whose mutable part is the bottom quiver.}\label{fig:two_marked_point_torus_freeze}
\end{figure}

We now show that surfaces with one marked point and one boundary component are locally-acyclic. We begin by finding coprime cluster variables.

\begin{lemma}\cite{cls}\label{lem:torus_loops}
Any non-contractible loop on a surface with one marked point and one boundary component is an element of the associated cluster algebra.
\end{lemma}

\begin{lemma}\label{lem:coprime_torus}
 Let  $T$ be a triangulation of a marked surface with exactly one boundary component and one marked point. For the arcs labeled $\alpha$ and $\beta$ in Figure~\ref{fig:gen_dread_torus} the cluster variables $x_\alpha$ and $x_{\beta}$ are coprime in $\A_\O^\bullet(Q_T)$ for any open set $\O \subset \spec(\ZZ(\hat Q_T))$ that is contained in the intersection $\bigcap_{\tau \in T} D_{y_\tau}$.  
\end{lemma}
\begin{proof} Let $I \subset \A_\O^\bullet(Q_T)$ be the ideal generated by $x_\alpha$ and $x_\beta$. 
Let $L, L_1, L_2$ and $\Gamma$ be the arcs or non-contractible loops shown in Figure~\ref{fig:gen_dread_torus}. By applying the smoothing operation twice to $\alpha$ and $L$ starting at the uppermost crossing and then applying Theorem~\ref{thm:smooth_relation} we see that 
\begin{equation}\label{eq:skein1}
x_\alpha x_L = Y_1 x_\beta + Y_2 x_{L_1} + Y_3 x_\Gamma .
\end{equation}
Now applying the smoothing operation to $\beta$ and $\Gamma$ and using Theorem~\ref{thm:smooth_relation} yields
\begin{equation}\label{eq:skein2}
x_\beta x_\Gamma = Y_4 x_\alpha^2 + Y_5 x_\alpha x_{L_2} + Y_6
\end{equation}
Recall the $Y_i$ are monomials in $\{y_\tau\}_{\tau \in T}$.
Using Equation~(\ref{eq:skein1}) and Equation~(\ref{eq:skein2}) we derive the relation
$$x_\alpha x_{\beta} x_L = Y_1 x_\beta^2 + Y_2 x_\beta x_{L_1} +Y_3Y_4 x_\alpha^2 + Y_3Y_5 x_\alpha x_{L_2} +Y_3 Y_6.$$
Since $L,L_1,$ and $L_2$ are in $\A_\O^\bullet(Q_T)$ by Lemma~\ref{lem:torus_loops} we see that $Y_3Y_6 \in I$. However, $Y_3Y_6$ is a unit in $\O_Y$ since $\O \subset\bigcap_{\tau \in T} D_{y_\tau}$ and thus $I = \A_\O^\bullet(Q_T)$. Therefore $x_\alpha$ and $x_\beta$ are coprime. 
\end{proof}

\begin{figure}
\includegraphics[scale=.7]{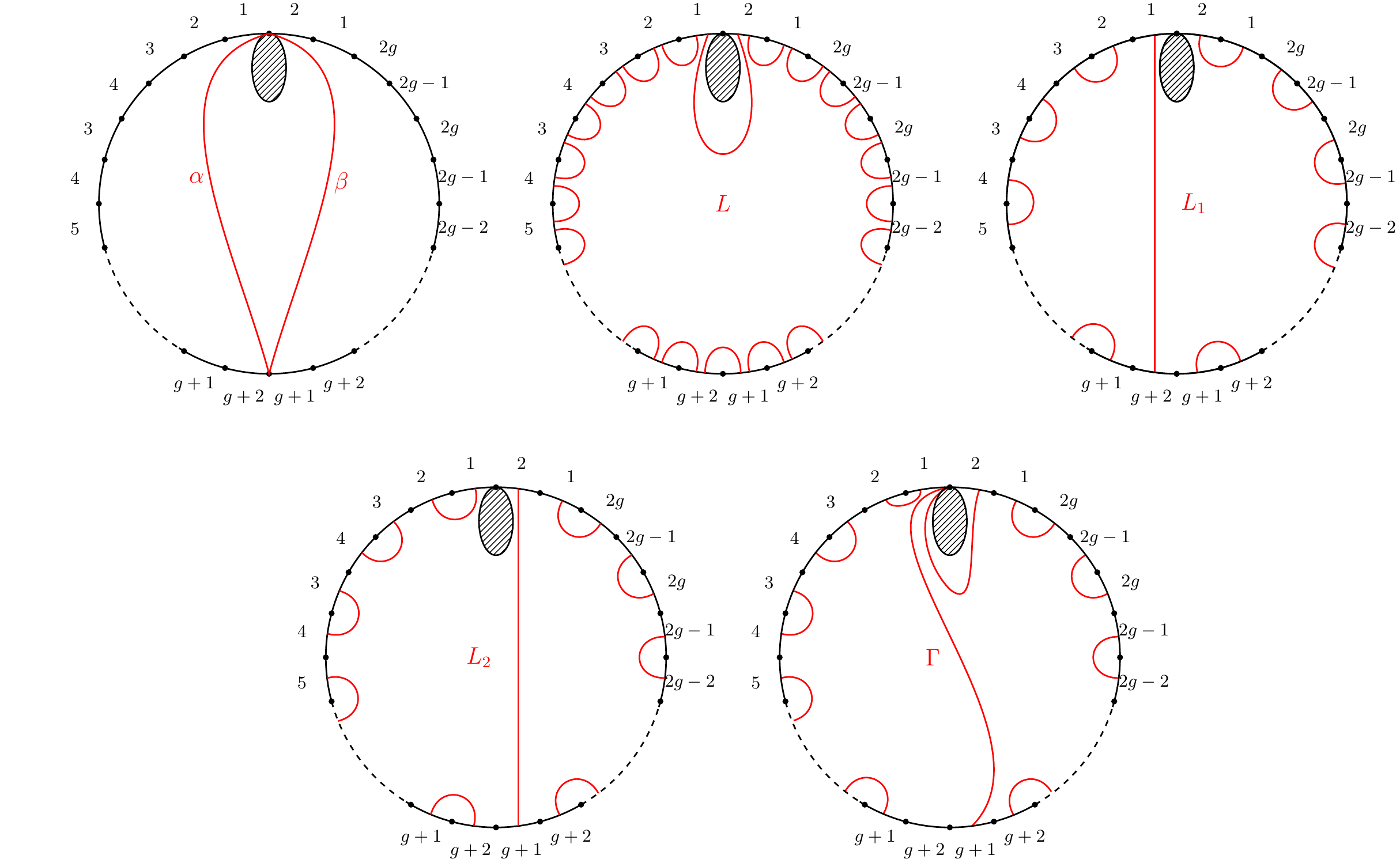}
\caption[Arcs and loops used in Lemma~\ref{lem:coprime_torus}]{Arcs and non-contractible loops for a genus $g$ marked surface with one boundary component and one marked point.}\label{fig:gen_dread_torus}
\end{figure}


\begin{theorem}\label{thm:LA_1punc_1bound}
 Let  $T$ be a triangulation of a marked surface with exactly one boundary component and one marked point containing the arcs $\alpha$ and $\beta$. The cluster algebra $\A_\O^\bullet(Q_T)$ is locally-acyclic for any open set $\O \subset \spec(\ZZ(\hat Q_T))$ that is contained in the intersection $\bigcap_{\tau \in T} D_{y_\tau}$.  
\end{theorem}
\begin{proof}
By Lemma~\ref{lem:coprime_torus} we have that $x_\alpha$ and $x_\beta$ are coprime and they both are elements of the same cluster that is associated with $T$. Let $\A^\dagger_\O$ be the cluster algebra obtained from the freezing of $x_\alpha$ and $\A^\ddagger_\O$ the cluster algebra obtained from the freezing of $x_\beta$. The cluster algebra $\A_\O^\dagger$ arises from the surface of genus one less than $\Sigma$ with two boundary components. One of the boundary components has one marked point and the other has two. Therefore $\A^\dagger_\O$ is locally-acyclic by Theorem~\ref{thm:boundaryLA}. By the symmetry of $\alpha$ and $\beta$ we see that the two freezings form isomorphic quivers so $\A^\ddagger_\O$ is also locally-acyclic. It follows from Lemma~\ref{lem:short_LA} that $\A_\O^\bullet$ is locally-acyclic. 
\end{proof}

We have finally proven for which cluster algebras from mutation-finite quivers are locally-acyclic cluster algebras. We finish with the proof of Theorem~\ref{thm:big1}.
\begin{proof}[Proof of Theorem~\ref{thm:big1}]
All cases for quivers of rank at least three follow immediately from Theorem~\ref{thm:LA_exceptional},  Theorem~\ref{thm:boundaryLA}, Theorem~\ref{thm:LA_closed_surface},  and Theorem~\ref{thm:LA_1punc_1bound}. In the rank 2 case we have that our quiver is acyclic so its cluster algebra is locally-acyclic by Theorem~\ref{thm:acyclic_LA}. The second claim is just restating a part of Theorem~\ref{thm:LA_exceptional} and Lemma~\ref{lem:1punc_no_au}.
\end{proof}

\section{Counterexample  to Conjecture~\ref{conj:false}}\label{sec:counterexample}

In this section we give an example of a quiver which is locally-acyclic and whose cluster algebra equals the upper cluster algebra, but is can not be mutated to any quiver with a maximal green sequence. We also show that every quiver in the mutation class has a green-to-red sequence. This provides a counterexample to Conjecture~\ref{conj:false} and evidence for Conjecture~\ref{conj:big}. 

We begin by showing that $\A(Q_{ce})$ is locally-acyclic. 

\begin{lemma}\label{lem:ce_1}
The cluster algebra $\A(Q_{ce})$ corresponding to the quiver $Q_{ce}$ from Figure~\ref{fig:counterexample} is locally-acyclic. 
\end{lemma}
\begin{proof}
Since vertex 1 is a source vertex in the quiver we have that the cluster variables $x_1$ and $x_2$ form a covering pair and hence are coprime by Lemma~\ref{lem:cover_is_coprime}. In Figure~\ref{fig:ce_is_la} we see that both the freezing of $x_1$ and the freezing of $x_2$ produce acyclic cluster algebras. By Theorem~\ref{thm:acyclic_LA} we know that acyclic cluster algebras are locally-acyclic so by Lemma~\ref{lem:short_LA} we have the desired result. 
\end{proof}
\begin{figure}
\includegraphics[scale=.8]{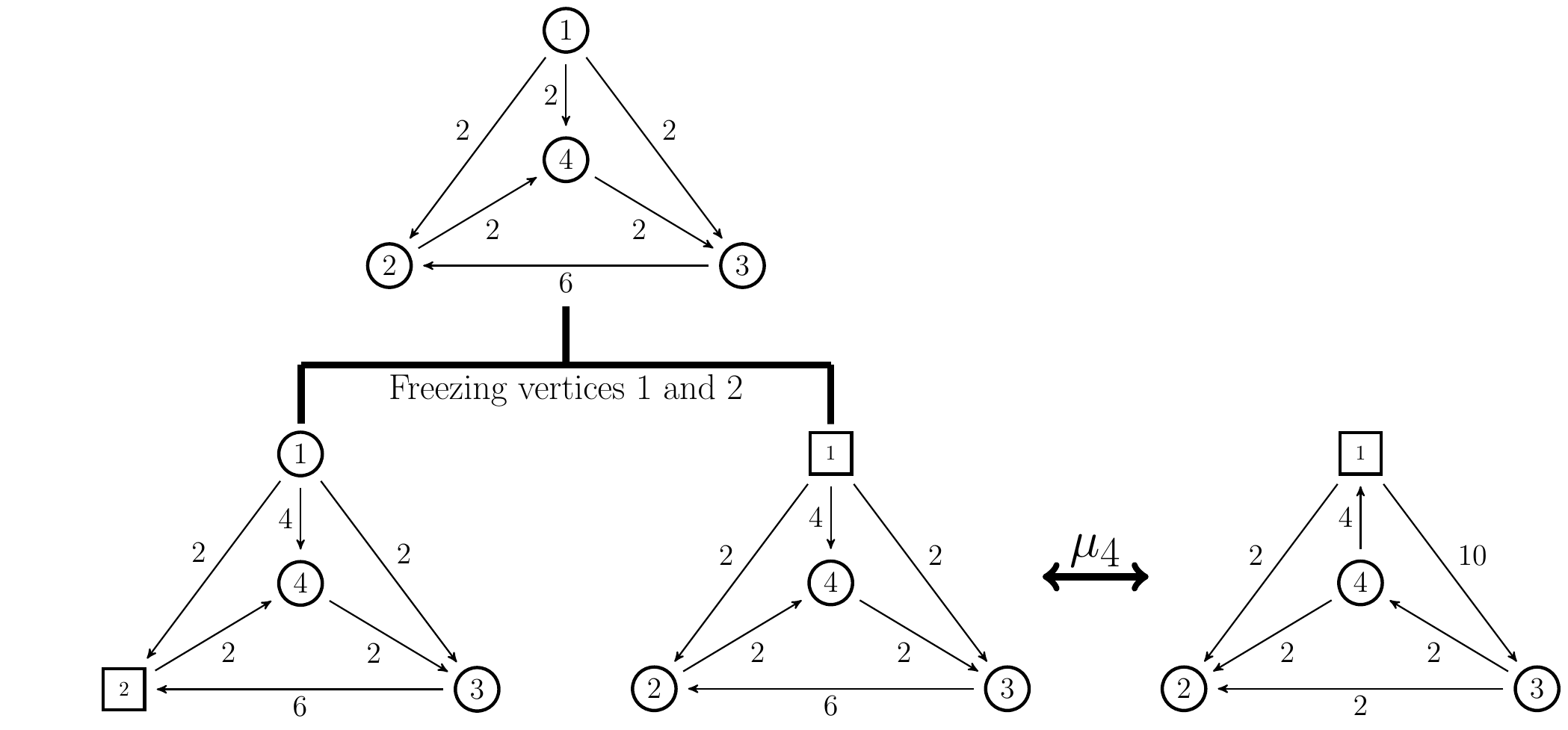}
\caption{The freezing of $Q_{ce}$ at vertices 1 and 2.}\label{fig:ce_is_la}
\end{figure}

We now work towards the claims about maximal green sequences and green-to-red sequences.

\begin{theorem}\cite[Theorem 4]{Brustle2017}\label{thm:bad_head}
At each step in a maximal green sequence the mutation occurs at a vertex which is not the head of a multiple arrow. 
\end{theorem}

\begin{theorem}{\cite[Lemma 1.4.1]{muller}}\label{thm:subquiver}
If a quiver admits a maximal green sequence, then any induced subquiver admits a maximal green sequence. 
\end{theorem}

Recall that $Q_1$ denotes the set of arrows of a quiver $Q$. 

\begin{corollary}\label{cor:bad_cycle}
If the vertices $i_1, \dots, i_\ell$ of a quiver $Q$ form an oriented cycle and\\ $\left|\{a \in Q_1 |  a = i_j\rightarrow i_{j+1} \}\right| \geq 2$ for $j = 1,\dots,\ell$ then $Q$ has no maximal green sequence.  
\end{corollary}
\begin{proof}
Note the induced subquiver of $Q$ on vertices $i_1, \dots, i_\ell$ has no maximal green sequence by Theorem~\ref{thm:bad_head} since every vertex is the head of a multiple arrow. The claim then follows from Theorem~\ref{thm:subquiver}. 
\end{proof}

The following observation was found in lecture notes of Philipp Lampe \cite{lampeLecture}. He attributes the result to Jan Schr\"oer. The proof is straightforward from the definition of mutation and can be found \emph{loc. cit.}
\begin{lemma}\cite[Exercise 2.4]{lampeLecture}\label{lem:lampe}
The greatest common divisor of the entries in a column of an exchange matrix $B = (b_{ij})$ is invariant under mutation. That is $$\gcd(b_{ij} : i \in [n]) = \gcd(b_{ij}':i \in [n]),$$ for $\mu_k(B) = B' = (b_{ij}')$ with $k \in [n].$
\end{lemma}

We introduce the following definitions in order to prove that $Q_{ce}$ cannot be mutated to an acyclic quiver. 

\begin{definition}
A \hdef{diagram} for a quiver $Q$, denoted $\Gamma(Q)$ is the directed graph whose vertex set is identical to that of $Q$ and whose set of arrows is the set underlying the multiset of arrows of $Q$. 

Let $Q$ be a quiver and let $A$ be the set of arrows of $\Gamma(Q)$. We say that $Q$ \hdef{admits an admissible coloring} if there exists a function $\alpha: \Gamma(Q)_1 \rightarrow \mathbb{Z}/2\mathbb{Z}$ such that 

\begin{enumerate}
\item For each oriented cycle, $C,$ of $\Gamma(Q)$ we have $\displaystyle\sum_{a \in C} \alpha(a) = 1 $. 
\item For each non-oriented cycle, $NC,$ of $\Gamma(Q)$ we have $\displaystyle\sum_{a \in NC} \alpha(a) = 0$. 
\end{enumerate}  
\end{definition}
\begin{theorem}\label{thm:admissible}
Let $Q$ be a quiver. If $\mut(Q)$ contains an acyclic quiver then $Q$ admits an admissible coloring. 
\end{theorem}
\begin{proof}
Suppose $\mut(Q)$ contains an acyclic quiver and let $B = (b_{ij})$ be the exchange matrix for $Q$. Theorem 1.2 of \cite{Seven2014b} provides the existence of an admissible quasi-Cartan matrix for $B$. This is a symmetric matrix of the form $C = (c_{ij})$ with the properties that \begin{enumerate}
\item $|c_{ij}| = |b_{ij}|$ for $i\neq j$
\item $c_{ii} = 2$ for $i = 1, \dots, n.$
\item Every oriented cycle $Z$ in $\Gamma(Q)$ has the property that the cardnality of the set\\ $\{c_{ij} > 0 | i \rightarrow j \in Z\}$ is odd.
\item Every non-oriented cycle $Z$ in $\Gamma(Q)$ has the property that the cardnality of the set\\ $\{c_{ij} > 0 | i \rightarrow j \in Z\}$ is even.
\end{enumerate} 
Therefore the function $\alpha: \Gamma(Q)_1 \rightarrow \Z/2\Z$ defined by 
$$ \alpha( i \rightarrow j ) = \left\{ \begin{array}{cc}
1 &\text{if } c_{ij}>0 \\ 0 &\text{otherwise} 
\end{array} \right. $$ shows that $Q$ admits an admissible coloring by properties $(3)$ and $(4)$ of admissible quasi-Cartan matrices. 
\end{proof}

\begin{corollary}\label{cor:no_acyclic}
There is no acyclic quiver in $\mut(Q_{ce})$.
\end{corollary}
\begin{proof}
By Theorem~\ref{thm:admissible} It suffices to show that $Q_{ce}$ does not admit an admissible coloring. The diagram $\Gamma(Q_{ce})$ is provided in Figure~\ref{fig:diagram}. 

Let $\alpha$ be function from the set of arrows to $\mathbb{Z}/2\mathbb{Z}$. Denote the images of the arrows under $\alpha$ by the labels shown in Figure~\ref{fig:diagram}.

There are three non-oriented cycles $1 \stackrel a \rightarrow 2 \stackrel b \rightarrow 4 \stackrel c \leftarrow 1,$
 and $1 \stackrel a \rightarrow 2 \stackrel e \leftarrow 4\stackrel f \leftarrow 1$,
 and $1 \stackrel c \rightarrow 4 \stackrel d \rightarrow 3 \stackrel f \leftarrow 1$ in $\Gamma(Q_{ce})$. 
Assume that $f$ satisfies the second condition of an admissible coloring. That is, $$a + b + c  = a+e + f = c + d + f = 0.$$ Adding these three equations together implies that $b+d+e = 0$. However, $2 \stackrel b \rightarrow  4\stackrel d \rightarrow  3\stackrel e \rightarrow 2$ is an oriented cycle in $\Gamma(Q_{ce})$ so the first condition of an admissible coloring cannot be met. Therefore it is impossible for $Q_{ce}$ to admit an admissible coloring. 

\end{proof}
\begin{figure}
\begin{tikzpicture}[scale = 0.8]
\node[vertex] (A) at (0,0){2};
\node[vertex] (B) at (6,0){3};
\node[vertex] (C) at (3,5){1};
\node[vertex] (D) at (3,1.8){4};

\draw[qarrow] (C) to node [above left]{ a} (A); 
\draw[qarrow] (C) to node [above right]{ f} (B); 
\draw[qarrow] (A) to node [below right]{ b} (D); 
\draw[qarrow] (D) to node [below left]{ d} (B); 
\draw[qarrow] (B) to node [below]{ e} (A); 
\draw[qarrow] (C) to node [left]{ c} (D);

\end{tikzpicture}
\caption{The diagram for $Q_{ce}$ with labeled arrows.}\label{fig:diagram}
\end{figure}
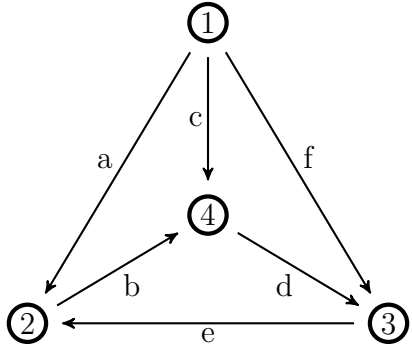

\begin{corollary}\label{cor:ce_1}
No quiver in $\mut(Q_{ce})$ has a maximal green sequence. 
\end{corollary}
\begin{proof}
Consider the exchange matrix $$B_{ce} = \begin{pmatrix}
0 & 2 & 2 & 4 \\
-2 & 0 & -6 & 2 \\
-2 & 6 & 0 & -2 \\
-4 & -2 & 2 & 0
\end{pmatrix}$$ corresponding to the quiver $Q_{ce}$.
It is clear that the $\gcd$ of every column of $B_{ce}$ is equal to 2. It follows from Lemma~\ref{lem:lampe} that no quiver in the mutation class of $Q_{ce}$ contains a single arrow since then the $\gcd$ of a column would be equal to 1. Therefore every non-acyclic quiver in $\mut(Q_{ce})$ has no maximal green sequence by Corollary~\ref{cor:bad_cycle}. The result now follows from Corollary~\ref{cor:no_acyclic}.
\end{proof}

\begin{lemma}{\cite[Corollary 19]{muller2}}\label{lem:g2r-inv}
If a quiver $Q$ admits a green-to-red sequence then any quiver $Q' \in \mut(Q)$ has a green-to-red sequence. 
\end{lemma}

\begin{corollary}\label{cor:ce_2}
Every quiver in $\mut(Q_{ce})$ has a green-to-red sequence.
\end{corollary}
\begin{proof}
It is straightforward to check that the mutation sequence $\mgsi=(1, 4, 3, 4, 2, 4)$ is a green-to-red sequence for $Q_{ce}$ so the claim follows from Lemma~\ref{lem:g2r-inv}. 
\end{proof}

\bibliographystyle{abbvmend}
\bibliography{library}

\begin{thebibliography}{10}

\bibitem{Alim2013a}
M.~Alim, S.~Cecotti, C.~C{\'{o}}rdova, S.~Espahbodi, A.~Rastogi, and C.~Vafa.
\newblock {BPS Quivers and Spectra of Complete N = 2 Quantum Field Theories}.
\newblock {\em Communications in Mathematical Physics}, 323(3):1185--1227,
  2013.
\newblock \href {http://arxiv.org/abs/arXiv:1109.4941v2}
  {\path{arXiv:arXiv:1109.4941v2}}, \href
  {http://dx.doi.org/10.1007/s00220-013-1789-8}
  {\path{doi:10.1007/s00220-013-1789-8}}.

\bibitem{Benito2015}
A.~Benito, G.~Muller, J.~Rajchgot, and K.~E. Smith.
\newblock {Singularities of locally acyclic cluster algebras}.
\newblock {\em Algebra and Number Theory}, 9(4):913--936, 2015.
\newblock \href {http://arxiv.org/abs/1404.4399} {\path{arXiv:1404.4399}},
  \href {http://dx.doi.org/10.2140/ant.2015.9.913}
  {\path{doi:10.2140/ant.2015.9.913}}.

\bibitem{cluster3}
A.~Berenstein, S.~Fomin, and A.~Zelevinsky.
\newblock {Cluster algebras III: Upper bounds and double Bruhat cells}.
\newblock {\em Duke Mathematical Journal}, 126(1):1--52, 2005.
\newblock \href {http://arxiv.org/abs/math/0305434}
  {\path{arXiv:math/0305434}}, \href
  {http://dx.doi.org/10.1215/S0012-7094-04-12611-9}
  {\path{doi:10.1215/S0012-7094-04-12611-9}}.

\bibitem{brustle}
T.~Br{\"{u}}stle, G.~Dupont, and M.~P{\'{e}}rotin.
\newblock {On maximal green sequences}.
\newblock {\em International Mathematics Research Notices},
  2014(16):4547--4586, 2014.
\newblock \href {http://arxiv.org/abs/1205.2050} {\path{arXiv:1205.2050}},
  \href {http://dx.doi.org/10.1093/imrn/rnt075}
  {\path{doi:10.1093/imrn/rnt075}}.

\bibitem{Brustle2017}
T.~Br{\"{u}}stle, S.~Hermes, K.~Igusa, and G.~Todorov.
\newblock {Semi-invariant pictures and two conjectures on maximal green
  sequences}.
\newblock {\em Journal of Algebra}, 473:80--109, mar 2017.
\newblock \href {http://dx.doi.org/10.1016/J.JALGEBRA.2016.10.025}
  {\path{doi:10.1016/J.JALGEBRA.2016.10.025}}.

\bibitem{bucher}
E.~Bucher.
\newblock {Maximal Green Sequences for Cluster Algebras Associated to
  Orientable Surfaces with Empty Boundary}.
\newblock {\em Arnold Mathematical Journal}, 2(4):487--510, dec 2016.
\newblock \href {http://dx.doi.org/10.1007/s40598-016-0057-3}
  {\path{doi:10.1007/s40598-016-0057-3}}.

\bibitem{BS}
E.~Bucher and J.~Machacek.
\newblock {Reddening sequences for Banff quivers and the class $\mathcal{P}$}.
\newblock jul 2018.
\newblock \href {http://arxiv.org/abs/1807.03359} {\path{arXiv:1807.03359}}.

\bibitem{BMS}
E.~Bucher, J.~Machacek, and M.~Shapiro.
\newblock {Upper cluster algebras and choice of ground ring}.
\newblock feb 2018.
\newblock \href {http://arxiv.org/abs/1802.04835} {\path{arXiv:1802.04835}}.

\bibitem{Bucher2017}
E.~Bucher and M.~R. Mills.
\newblock {Maximal green sequences for cluster algebras associated with the
  n-torus with arbitrary punctures}.
\newblock {\em Journal of Algebraic Combinatorics}, pages 1--12, 2017.
\newblock \href {http://arxiv.org/abs/1503.06207} {\path{arXiv:1503.06207}},
  \href {http://dx.doi.org/10.1007/s10801-017-0778-y}
  {\path{doi:10.1007/s10801-017-0778-y}}.

\bibitem{cls}
I.~Canakci, K.~Lee, and R.~Schiffler.
\newblock {On cluster algebras from unpunctured surfaces with one marked
  point}.
\newblock {\em Proceedings of the American Mathematical Society, Series B},
  2(3):35--49, 2015.
\newblock \href {http://arxiv.org/abs/1407.5060} {\path{arXiv:1407.5060}},
  \href {http://dx.doi.org/10.1090/bproc/21} {\path{doi:10.1090/bproc/21}}.

\bibitem{derksen}
H.~Derksen, J.~Weyman, and A.~Zelevinsky.
\newblock {Quivers with potentials and their representations I: Mutations}.
\newblock {\em Selecta Mathematica, New Series}, 14(1):59--119, 2008.
\newblock \href {http://arxiv.org/abs/0704.0649} {\path{arXiv:0704.0649}},
  \href {http://dx.doi.org/10.1007/s00029-008-0057-9}
  {\path{doi:10.1007/s00029-008-0057-9}}.

\bibitem{FST}
A.~Felikson and M.~Shapiro.
\newblock {Skew-symmetric cluster algebras of finite mutation type}.
\newblock {\em arXiv}, 5916:49, 2008.
\newblock \href {http://arxiv.org/abs/0811.1703} {\path{arXiv:0811.1703}}.

\bibitem{FockGoncharov2006}
V.~Fock and A.~Goncharov.
\newblock {Moduli spaces of local systems and higher Teichm{\"{u}}ller theory}.
\newblock {\em Publications math{\'{e}}matiques de l'IH{\'{E}}S},
  103(1):1--211, jun 2006.
\newblock \href {http://dx.doi.org/10.1007/s10240-006-0039-4}
  {\path{doi:10.1007/s10240-006-0039-4}}.

\bibitem{fomin}
S.~Fomin, M.~Shapiro, and D.~Thurston.
\newblock {Cluster algebras and triangulated surfaces. Part I: Cluster
  complexes}.
\newblock {\em Acta Mathematica}, 201(1):83--146, 2008.
\newblock \href {http://arxiv.org/abs/math/0608367}
  {\path{arXiv:math/0608367}}, \href
  {http://dx.doi.org/10.1007/s11511-008-0030-7}
  {\path{doi:10.1007/s11511-008-0030-7}}.

\bibitem{cluster1}
S.~Fomin and A.~Zelevinsky.
\newblock {Cluster algebras I: Foundations}.
\newblock {\em Journal of the American Mathematical Society}, 15(2):497--529,
  2001.
\newblock \href {http://arxiv.org/abs/math/0104151}
  {\path{arXiv:math/0104151}}, \href
  {http://dx.doi.org/10.1090/S0894-0347-01-00385-X}
  {\path{doi:10.1090/S0894-0347-01-00385-X}}.

\bibitem{cluster2}
S.~Fomin and A.~Zelevinsky.
\newblock {Cluster algebras II: Finite type classification}.
\newblock {\em Inventiones Mathematicae}, 0070685(1):1--50, 2003.
\newblock \href {http://arxiv.org/abs/arXiv:math/0208229v2}
  {\path{arXiv:arXiv:math/0208229v2}}, \href
  {http://dx.doi.org/10.1007/s00222-003-0302-y}
  {\path{doi:10.1007/s00222-003-0302-y}}.

\bibitem{cluster4}
S.~Fomin and A.~Zelevinsky.
\newblock {Cluster algebras IV: Coefficients}.
\newblock {\em Compositio Mathematica}, 143(1):112--164, 2007.
\newblock \href {http://arxiv.org/abs/math/0602259}
  {\path{arXiv:math/0602259}}, \href
  {http://dx.doi.org/10.1112/S0010437X06002521}
  {\path{doi:10.1112/S0010437X06002521}}.

\bibitem{Grabowski}
J.~E. Grabowski.
\newblock {Graded cluster algebras}.
\newblock {\em Journal of Algebraic Combinatorics}, 42(4):1111--1134, dec 2015.
\newblock \href {http://dx.doi.org/10.1007/s10801-015-0619-9}
  {\path{doi:10.1007/s10801-015-0619-9}}.

\bibitem{Gross2015}
M.~Gross, P.~Hacking, and S.~Keel.
\newblock {Birational geometry of cluster algebras}.
\newblock {\em Algebraic Geometry}, 2(2):137--175, 2015.
\newblock \href {http://arxiv.org/abs/1309.2573} {\path{arXiv:1309.2573}},
  \href {http://dx.doi.org/10.14231/AG-2015-007}
  {\path{doi:10.14231/AG-2015-007}}.

\bibitem{gross}
M.~Gross, P.~Hacking, S.~Keel, and M.~Kontsevich.
\newblock {Canonical bases for cluster algebras}.
\newblock {\em arXiv:1411.1394}, 2014.
\newblock \href {http://arxiv.org/abs/1411.1394} {\path{arXiv:1411.1394}}.

\bibitem{keller2}
B.~Keller.
\newblock {Quiver mutation and combinatorial DT-invariants}.
\newblock {\em DMTCS proceedings}, (1995):9--20, 2013.

\bibitem{Kontsevich2008}
M.~Kontsevich and Y.~Soibelman.
\newblock {Stability structures, motivic Donaldson-Thomas invariants and
  cluster transformations}.
\newblock {\em arXiv:0811.2435}, page 148, 2008.
\newblock \href {http://arxiv.org/abs/0811.2435} {\path{arXiv:0811.2435}}.

\bibitem{ladkani}
S.~Ladkani.
\newblock {On cluster algebras from once punctured closed surfaces}.
\newblock {\em Arxiv preprint arXiv:1310.4454}, oct 2013.
\newblock \href {http://arxiv.org/abs/1310.4454} {\path{arXiv:1310.4454}}.

\bibitem{lampeLecture}
P.~Lampe.
\newblock {Lecture notes}.
\newblock
  \url{http://www.maths.dur.ac.uk/users/philipp.b.lampe/LectureNotes/cluster.pdf}.

\bibitem{lawson}
J.~W. Lawson.
\newblock {Minimal Mutation-Infinite Quivers}.
\newblock {\em Experimental Mathematics}, 26(3):308--323, 2017.
\newblock \href {http://arxiv.org/abs/1505.01735} {\path{arXiv:1505.01735}},
  \href {http://dx.doi.org/10.1080/10586458.2016.1166353}
  {\path{doi:10.1080/10586458.2016.1166353}}.

\bibitem{leeschiffler}
K.~Lee and R.~Schiffler.
\newblock {Positivity for cluster algebras}.
\newblock {\em Annals of Mathematics}, 182(1):73--125, 2015.
\newblock \href {http://arxiv.org/abs/1306.2415} {\path{arXiv:1306.2415}},
  \href {http://dx.doi.org/10.4007/annals.2015.182.1.2}
  {\path{doi:10.4007/annals.2015.182.1.2}}.

\bibitem{MM}
J.~P. Matherne and G.~Muller.
\newblock {Computing Upper Cluster Algebras}.
\newblock {\em International Mathematics Research Notices},
  2015(11):3121--3149, mar 2014.
\newblock \href {http://dx.doi.org/10.1093/imrn/rnu027}
  {\path{doi:10.1093/imrn/rnu027}}.

\bibitem{mills}
M.~R. Mills.
\newblock {Maximal green sequences for quivers of finite mutation type}.
\newblock {\em Advances in Mathematics}, 319:182--210, 2017.
\newblock \href {http://arxiv.org/abs/1606.03799} {\path{arXiv:1606.03799}},
  \href {http://dx.doi.org/10.1016/j.aim.2017.08.019}
  {\path{doi:10.1016/j.aim.2017.08.019}}.

\bibitem{muller}
G.~Muller.
\newblock {Locally acyclic cluster algebras}.
\newblock {\em Advances in Mathematics}, 233(1):207--247, 2013.
\newblock \href {http://arxiv.org/abs/1111.4468} {\path{arXiv:1111.4468}},
  \href {http://dx.doi.org/10.1016/j.aim.2012.10.002}
  {\path{doi:10.1016/j.aim.2012.10.002}}.

\bibitem{MullerAU}
G.~Muller.
\newblock {$\mathcal A= \mathcal U$ for locally acyclic cluster algebras}.
\newblock {\em SIGMA Symmetry Integrability Geom. Methods Appl.}, 10:8, Paper
  094, 2014.
\newblock \href {http://arxiv.org/abs/1308.1141} {\path{arXiv:1308.1141}},
  \href {http://dx.doi.org/10.3842/SIGMA.2014.094}
  {\path{doi:10.3842/SIGMA.2014.094}}.

\bibitem{muller2}
G.~Muller.
\newblock {The existence of a maximal green sequence is not invariant under
  quiver mutation}.
\newblock {\em Electronic Journal of Combinatorics}, 23(2), 2016.
\newblock \href {http://arxiv.org/abs/1503.04675} {\path{arXiv:1503.04675}}.

\bibitem{MS}
G.~Muller and D.~E. Speyer.
\newblock {Cluster algebras of Grassmannians are locally acyclic}.
\newblock {\em Proceedings of the American Mathematical Society},
  144(8):3267--3281, mar 2016.
\newblock \href {http://dx.doi.org/10.1090/proc/13023}
  {\path{doi:10.1090/proc/13023}}.

\bibitem{MSW}
G.~Musiker, R.~Schiffler, and L.~Williams.
\newblock {Positivity for cluster algebras from surfaces}.
\newblock {\em Advances in Mathematics}, 227(6):2241--2308, aug 2011.
\newblock \href {http://dx.doi.org/10.1016/J.AIM.2011.04.018}
  {\path{doi:10.1016/J.AIM.2011.04.018}}.

\bibitem{MSW2}
G.~Musiker, R.~Schiffler, and L.~Williams.
\newblock {Bases for cluster algebras from surfaces}.
\newblock {\em Compositio Mathematica}, 149(02):217--263, feb 2013.
\newblock \href {http://dx.doi.org/10.1112/S0010437X12000450}
  {\path{doi:10.1112/S0010437X12000450}}.

\bibitem{MW}
G.~Musiker and L.~Williams.
\newblock {Matrix formulae and skein relations for cluster algebras from
  surfaces}.
\newblock {\em International Mathematics Research Notices},
  2013(13):2891--2944, 2013.
\newblock \href {http://arxiv.org/abs/1108.3382} {\path{arXiv:1108.3382}},
  \href {http://dx.doi.org/10.1093/imrn/rns118}
  {\path{doi:10.1093/imrn/rns118}}.

\bibitem{Seven2014b}
A.~I. Seven.
\newblock {Cluster algebras and symmetric matrices}.
\newblock {\em Proceedings of the American Mathematical Society},
  143(2):469--478, oct 2014.
\newblock \href {http://dx.doi.org/10.1090/S0002-9939-2014-12252-0}
  {\path{doi:10.1090/S0002-9939-2014-12252-0}}.

\bibitem{seven}
A.~I. Seven.
\newblock {Maximal Green Sequences of Exceptional Finite Mutation Type
  Quivers}.
\newblock {\em SIGMA. Symmetry, Integrability and Geometry: Methods and
  Applications}, 10:089, jun 2014.
\newblock \href {http://arxiv.org/abs/1406.1072} {\path{arXiv:1406.1072}},
  \href {http://dx.doi.org/10.3842/SIGMA.2014.089}
  {\path{doi:10.3842/SIGMA.2014.089}}.

\bibitem{Yakimov}
M.~Yakimov and K.~Goodearl.
\newblock {Maximal green sequences for double Bruhat cells}.

\end{thebibliography}

\end{document}